\documentclass[a4paper,10pt,leqno]{amsart}
\title{Approximating $L^2$-invariants and homology growth}
\author{L\"uck, W.}
        \address{Mathematisches Institut der Universit\"at Bonn\\
                Endenicher Allee 60\\
                53115 Bonn, Germany}
         \email{wolfgang.lueck@him.uni-bonn.de}
          \urladdr{http://www.him.uni-bonn.de/lueck}
         \date{October 2012}

\usepackage{hyperref}
\usepackage{color}
\usepackage{pdfsync}
\usepackage{calc}
\usepackage{enumerate,amssymb}
\usepackage[arrow,curve,matrix,tips,2cell]{xy}
  \SelectTips{eu}{10} \UseTips
  \UseAllTwocells

\DeclareMathAlphabet\EuR{U}{eur}{m}{n}
\SetMathAlphabet\EuR{bold}{U}{eur}{b}{n}

\makeindex             

\theoremstyle{plain}
\newtheorem{theorem}{Theorem}[section]
\newtheorem{lemma}[theorem]{Lemma}
\newtheorem{proposition}[theorem]{Proposition}
\newtheorem{corollary}[theorem]{Corollary}
\newtheorem{conjecture}[theorem]{Conjecture}

\theoremstyle{definition}

\newtheorem{definition}[theorem]{Definition}
\newtheorem{example}[theorem]{Example}

\newtheorem{remark}[theorem]{Remark}
\newtheorem{notation}[theorem]{Notation}
\newtheorem{setup}[theorem]{Setup}

\newtheorem{question}[theorem]{Question}

{\catcode`@=11\global\let\c@equation=\c@theorem}

\newcommand{\comsquare}[8]                   
{\begin{CD}
#1 @>#2>> #3\\
@V{#4}VV @V{#5}VV\\
#6 @>#7>> #8
\end{CD}
}

\newcommand{\xycomsquare}[8]                   
{\xymatrix
{#1 \ar[r]^{#2} \ar[d]^{#4} &
#3 \ar[d]^{#5}  \\
#6\ar[r]^{#7} &
#8
}
}

\newcommand{\xycomsquareminus}[8]                      
{\xymatrix{#1 \ar[r]^-{#2} \ar[d]^-{#4} &
#3 \ar[d]^-{#5}  \\
#6\ar[r]^-{#7} &
#8
}
}

\newcommand{\caln}{{\mathcal N}}

\newcommand{\IC}{{\mathbb C}}

\newcommand{\IF}{{\mathbb F}}

\newcommand{\IN}{{\mathbb N}}

\newcommand{\IQ}{{\mathbb Q}}
\newcommand{\IR}{{\mathbb R}}

\newcommand{\IZ}{{\mathbb Z}}

\newcommand{\aut}{\operatorname{aut}}

\newcommand{\coker}{\operatorname{coker}}

\newcommand{\ev}{\operatorname{ev}}

\newcommand{\id}{\operatorname{id}}

\newcommand{\im}{\operatorname{im}}

\newcommand{\mg}{\operatorname{d}}

\newcommand{\pr}{\operatorname{pr}}

\newcommand{\rk}{\operatorname{rk}}

\newcommand{\Tor}{\operatorname{Tor}}
\newcommand{\tors}{\operatorname{tors}}

\newcommand{\pt}{\{\bullet\}}

\newcommand{\higherlim}[3]{{\setbox1=\hbox{\rm lim}
        \setbox2=\hbox to \wd1{\leftarrowfill} \ht2=0pt \dp2=-1pt
        \mathop{\vtop{\baselineskip=5pt\box1\box2}}
        _{#1}}^{#2}#3}

\newcommand{\version}[1]                    
{\begin{center} last edited on #1\\
last compiled on \today\\
name of texfile: \jobname
\end{center}
}

\newcounter{commentcounter}

\begin{document}

\typeout{----------------------------  l2approx.Tex  --------------------------}


\typeout{------------------------- Abstract ------------------------------------}

\begin{abstract}
  In this paper we consider the asymptotic behavior of invariants such as Betti
  numbers, minimal numbers of generators of singular homology, the order of the torsion subgroup of
  singular homology, and torsion invariants. We will show that all these
  vanish in the limit if the $CW$-complex under consideration fibers in a
  specific way. In particular we will show that all these vanish in the limit if
  one considers an aspherical closed  manifold which admits a non-trivial
  $S^1$-action or whose fundamental group contains an infinite normal elementary
  amenable subgroup. By considering classifying spaces we also get results for
  groups.
\end{abstract}

\maketitle


 \typeout{---------------------   Section 0: Introduction ----------------------}

\setcounter{section}{-1}
\section{Introduction}
\label{sec:introduction}

In this paper we consider the asymptotic behavior of invariants such as Betti
numbers, minimal numbers of generator of singular homology, the order of the
torsion subgroup of singular homology, and torsion invariants. We will use the
following setup.
\begin{setup}\label{set:inverse_systems}
  Let $G$ be a group together with an inverse system $\{G_i \mid i \in I\}$ of
  normal subgroups of $G$ directed by inclusion over the directed set $I$ such
  that $[G:G_i]$ is finite for all $i \in I$ and $\bigcap_{i \in I} G_i = \{1\}$. Let $K$ be a field.
\end{setup}

We want to study for a $G$-covering $\overline{X} \to X$ of a connected finite $CW$-complex $X$ the nets
\[\left(\frac{b_n(G_i\backslash X;K)}{[G:G_i]}\right)_{i \in I}, 
\left(\frac{\ln\left(\left|\tors(H_n(G_i\backslash \overline{X}))\right|\right)}{[G:G_i]}\right)_{i \in I},
\left(\frac{\mg(H_n(G_i\backslash \overline{X}))}{[G:G_i]}\right)_{i \in I},
\]
\[
\left(\frac{\mg(G_i)}{[G:G_i]}\right)_{i \in I}, \left(\frac{\rho^{(2)}(G_i\backslash \overline{X};\caln(\{1\}))}{[G:G_i]} \right)_{i \in I},
\quad \text{and}\quad 
\left(\frac{\rho^{\IZ}(G_i\backslash \overline{X})}{[G:G_i]} \right)_{i \in I},
\]

where $\mg$ denotes the minimal number of generators, $b_n(G_i\backslash X;K)$
the  $n$th Betti number with coefficients in the field $K$, $\tors(H_n(G_i\backslash \overline{X}))$
the torsion subgroup of the abelian group given by the $n$th singular homology with integer coefficients, 
$\rho^{(2)}(G_i\backslash \overline{X};\caln(\{1\}))$
the $L^2$-torsion (with respect to the trivial group) and $\rho^{\IZ}(G_i\backslash \overline{X})$ the integral torsion. 

Of particular interest is the sequence
$\frac{\ln\left(\left|\tors(H_n(G_i\backslash \overline{X}))\right|\right)}{[G:G_i]}$ for which most of the known results
are restricted to infinite cyclic fundamental groups (see for 
instance~\cite{Bergeron-Venkatesh(2010),Silver-Williams(2002),Silver-Williams(2004)}).
The other sequences have already been studied intensively in the literature, see for instance
\cite{Abert-Jaikin-Zapirain-Nikolov(2011),Abert-Nikolov(2012),Bergeron-Linnell-Lueck-Sauer(2012),%
Calegari-Emerton(2009bounds),Calegari-Emerton(2011),Clair-Whyte(2003),Ershof-Lueck(2012),%
Lackenby(2005expanders), Lackenby(2009propery_tau),%
Lackenby(2009_new_bound),Linnell-Lueck-Sauer(2011),Lueck(1994c),Osin(2011_rankgradient)}.

We will concentrate on cases, where we can show that these sequences converge to
zero, for instance, when $X$ is a closed aspherical manifold, $G = \pi_1(X)$,
$\Phi = \id_G$, $\overline{X} = \widetilde{X}$ and we assume either that $X$
carries a non-trivial $S^1$-action or $G$ contains a non-trivial elementary
amenable normal subgroup (see
Corollary~\ref{cor:Groups_containing_a_normal_infinite_nice_subgroups}). Thus we
can prove in this special case the Approximation Conjectures of
Subsection~\ref{subsec:Approximation_Conjectures} which present one of the basic
motivation for this paper.

If one considers $\overline{X} = EG$, one obtains statements about the group $G$
itself as explained in Subsection~\ref{subsec:Homological_growth_of_groups}.

We briefly discuss the related notions of rank gradient and cost in
Subsection~\ref{subsec:Exponent_gradient}.

This paper is financially supported by the Leibniz-Preis of the author. Many
thanks of the author go to Nicolas Bergeron and Roman Sauer for fruitful
discussions on the topic.



 \typeout{---------------------   Section 1: Statement of Results ----------------------}

\section{Statement of results}
\label{sec:statement_of_results}

\subsection{Main result}
\label{subsec:main_result}

Next we state our main technical result.
All other results are essentially consequences of it.
Explanations follow below.

\begin{theorem}[Fibrations]
  \label{the:fibrations}
  Let $F\xrightarrow{j} X \xrightarrow{f} B$ be a fibration of connected
  $CW$-complexes.    Consider a homomorphism
  $\phi \colon \pi_1(X) \to G$.  Let $p \colon \overline{X} \to X$ be 
  the associated $G$-covering.  Let $G_1(F) \subseteq \pi_1(F)$ be Gottlieb's
  subgroup of the fundamental group of $F$ (see~Definition~\ref{def:G_1(F)}.) 
  Suppose that the image of $G_1(F)$ under the
  composite $\phi \circ \pi_1(j) \colon \pi_1(F) \to G$ is infinite.

  If $d$ is a natural number such that the $(d+1)$-skeleton of $X$ is finite, then:

\begin{enumerate}

\item \label{the:fibrations:limit_of_mg}
We get for all $n \le d$
\[
\lim_{i\in I} \frac{\mg\bigl(H_n(G_i\backslash \overline{X})\bigr)}{[G:G_i]} 
= 0;
\]

\item \label{the:fibrations:limit_of_torsion_in_homology}
We get for all $n \le d$
\[
\lim_{i\in I} \frac{\ln\left(\left|\tors\bigl(H_n(G_i\backslash \overline{X})\bigr)\right|\right)}{[G:G_i]} 
= 0;
\]

\item \label{the:fibrations:L2-Betti}
We get for all $n \le d$
\[
b_n^{(2)}\bigl(\overline{X};\caln(G)\bigr) = \lim_{i \to \infty} \frac{b_n(G_i\backslash X;K)}{[G:G_i]} = 0;
\]

\item \label{the:fibrations:limit_of_torsion}
Suppose that $X$ is a  connected finite $CW$-complex. Then
\[
\lim_{i\in I} \frac{\rho^{(2)}\bigl(G_i\backslash \overline{X};\caln(\{1\})\bigr)}{[G:G_i]}
= \lim_{i\in I} \frac{\rho^{\IZ}\bigl(G_i\backslash \overline{X}\bigr)}{[G:G_i]} = 0;
\]

\item \label{the:fibrations:L2-torsion}
Suppose that  both $F$ and $B$ are connected finite $CW$-complexes and  that 
$\rho^{(2)}(\overline{F};\caln(H)) = 0$, where
$H$ is the image of the composite 
$\pi_1(F) \xrightarrow{\pi_1(j)} \pi_1(E) \xrightarrow{\phi} G$ and
$\overline{F}$ is the covering associated to the induced epimorphism 
$\pi_1(F) \to H$. Then the $L^2$-torsion $\rho^{(2)}\bigl(\overline{X};\caln(G)\bigr)$
is defined and satisfies
\[
\rho^{(2)}\bigl(\overline{X};\caln(G)\bigr) = 0.
\]

\end{enumerate}
\end{theorem}

\begin{remark}\label{rem_no_assumptions}
  We emphasize that we do \emph{not} require in Theorem~\ref{the:fibrations} that
  the fibre transport of the fibration $f \colon X \to B$ is trivial or that
  $F$, $B$ or $G_1(F)$ satisfy any finiteness assumptions except in
  assertion~\eqref{the:fibrations:L2-torsion}. 
\end{remark}

\begin{remark}\label{rem:redundant}
Some of the assertions appearing in Theorem~\ref{the:fibrations} imply one another.
For instance , we have the following facts which are all consequences
of the Universal Coefficients Theorem and 
Lemma~\ref{lem:properties_of_mg}

If we have $\lim_{i\in I} \frac{\mg(H_n(G_i\backslash \overline{X}))}{[G:G_i]} = 0$
for all $n \le d$, then we get for all $n \le d$
\[
\lim_{i\in I} \frac{b_n(G_i\backslash \overline{X};K)}{[G:G_i]} = 0.
\]

If $n \le d$ and we have  both $\lim_{i\in I} \frac{\ln(|\tors(H_n(G_i\backslash \overline{X}))|)}{[G:G_i]} = 0$
and $\lim_{i\in I} \frac{b_n(G_i\backslash \overline{X};\IQ)}{[G:G_i]} = 0$, then 
\[
\lim_{i\in I} \frac{\mg\bigl(H_n(G_i\backslash \overline{X})\bigr)}{[G:G_i]} = 0.
\]
\end{remark}

For the notions of $L^2$-Betti number
  $b_n^{(2)}\bigl(\overline{X};\caln(G)\bigr)$ and $L^2$-torsion
  $\rho^{(2)}\bigl(\overline{X};\caln(G)\bigr)$ we refer for instance
  to~\cite{Lueck(2002)}.

\begin{definition}[Integral torsion]\label{def:integral_torsion}
  Define for a finite $\IZ$-chain complex $D_*$ its \emph{integral torsion}
  \begin{eqnarray*}
    \rho^{\IZ}(D_*) 
    & := & 
    \sum_{n \ge 0} (-1)^n \cdot \ln\left(\bigl|\tors(H_{n}(D_*))\bigr|\right) 
    \quad \in \IR,
  \end{eqnarray*}
  where $\bigl|\tors(H_n(D_*))\bigr|$ is the order of 
  the torsion subgroup of the finitely generated abelian group
  $H_n(D_*)$.

  Given a finite $CW$-complex $X$, define its \emph{integral torsion}
  $\rho^{\IZ}(X)$ by $\rho^{\IZ}(C_*(X))$, where $C_*(X)$ is its cellular
  $\IZ$-complex.
\end{definition}

\begin{remark}[Integral torsion and Milnor's torsion] 
\label{rem:Milnor_torsion}
Let $C_*$ be a finite free $\IZ$-chain complex. Fix for each $n \ge 0$ a
$\IZ$-bases for $C_n$ and for $H_n(C)/\tors(H_n(C))$. They induce $\IQ$-basis
for $\IQ \otimes_{\IZ} C_n$ and 
$H_n\bigl(\IQ \otimes_{\IZ} C_*) 
\cong \IQ \otimes_{\IZ} \bigl(H_n(C)/\tors(H_n(C)\bigr)\bigr)$.
Then the torsion in the sense of Milnor~\cite[page~365]{Milnor(1966)} is $\rho^{\IZ}(C_*)$.
\end{remark}

\begin{definition}[Gottlieb's subgroup of the fundamental group]\label{def:G_1(F)}
  Let $F = (F,x_0)$ be a pointed $CW$-complex. Let 
  $G_1(F,x_0)   \subseteq \pi_1(F,x_0)$ be the subgroup of elements 
  $[w] \in   \pi_1(F,x_0)$ for which there exists a homotopy 
  $h \colon F \times   [0,1] \to F$ 
   such that $h(x,0) = h(x,1) = x$ holds for $x \in F$ and
  the loop at $x_0$ given by $t \mapsto H(x_0,t)$ represents $[w]$.
\end{definition}

\subsection{Connected Lie groups as fibers}
\label{subsec:Connected_Lie_groups_as_fibers}

We give an example, where the conditions about $G_1(F)$ are automatically satisfied.

\begin{corollary}[Connected Lie group as fiber] 
\label{cor:connected_Lie_groups_as_fibers}
Let $L \xrightarrow{j} X \xrightarrow{f} B$ be a fibration with a connected Lie group 
$L$ as fiber and connected $CW$-complexes $B$ as basis and $X$ as total space. Consider a homomorphism
$\phi \colon \pi_1(X) \to G$  such that
the image of $\pi_1(L) \xrightarrow{\pi_1(j)} \pi_1(X) \xrightarrow{\phi} G$ is 
infinite.  Then all assumptions
in Theorem~\ref{the:fibrations} are satisfied   and 
all the assertions~\eqref{the:fibrations:limit_of_mg},~%
\eqref{the:fibrations:limit_of_torsion_in_homology},~%
\eqref{the:fibrations:L2-Betti},~%
\eqref{the:fibrations:limit_of_torsion}
and~\eqref{the:fibrations:L2-torsion} hold.
\end{corollary}
\begin{proof}
Lemma~\ref{lem:characterizing_G_1(F)}~\eqref{lem:characterizing_G_1(F):example}
implies $G_1(L) = \pi_1(L)$. Because of Theorem~\ref{the:fibrations}
it remains to show
$\rho^{(2)}(\overline{L};\caln(H)) = 0$, where $H$ is the image of
$\pi_1(L) \xrightarrow{\pi_1(j)} \pi_1(X) \xrightarrow{\phi} G$ and
$\overline{L} \to L$ is the covering associated to the induced epimorphism
$\pi_1(L)  \to H$.

Let $T \subseteq L$ be a maximal
torus. Then we have the fibration $T \xrightarrow{i} L \xrightarrow{p} L/T$. Again by
Lemma~\ref{lem:characterizing_G_1(F)}~\eqref{lem:characterizing_G_1(F):example}
we have $G_1(T) = \pi_1(T)$. Hence by 
Theorem~\ref{the:fibrations}~\eqref{the:fibrations:L2-torsion}  it suffices to show
$\rho^{(2)}(\overline{T};\caln(H)) = 0$, where $\overline{T} \to T$
is the covering associated to the epimorphism 
$\pi_1(T) \xrightarrow{\pi_1(i)} \pi_1(L)  \xrightarrow{\pi_1(j)}  \pi_1(X) \xrightarrow{\phi} H$. 
(This is an epimorphism as  $\pi_1(T) \to \pi_1(L)$ is surjective.). 
Since $H$ is by assumption infinite,
we can write $H$ as a product of subgroups $\IZ \times A$ and
$\overline{T}$ as the product of the $\IZ$-space $\widetilde{S^1}$ and
of a finite free $A$-$CW$-complex $Y$. Then we conclude 
from~\cite[Theorem~3.93 on page~161]{Lueck(2002)}
\begin{eqnarray*}
\rho^{(2)}(\overline{T};\caln(H))
& = & 
\rho^{(2)}\bigl(\widetilde{S^1}\times Y;\caln(\IZ \times A)\bigr)
\\
& = & 
\rho^{(2)}\bigl(\widetilde{S^1};\caln(\IZ)\bigr) \cdot \chi(A \backslash Y)
\\
& = & 0,
\end{eqnarray*}
since $\rho^{(2)}\bigl(\widetilde{S^1};\caln(\IZ)\bigr)$ vanishes
by~\cite[(3.24) on page~136]{Lueck(2002)}.
\end{proof}

\begin{example} \label{exa:two-connected_base}
Let $L \xrightarrow{j} X \xrightarrow{f} B$ be a fibration with a connected Lie group 
$L$ as fiber and connected $CW$-complexes $B$ as basis and $X$ as total space. Suppose that $\pi_2(B)$ vanishes
and that $\pi_1(L)$ is infinite.
Consider the case $G = \pi_1(X)$, $\phi = \id_G$ and $\overline{X} = \widetilde{X}$.
Then all assumptions
in Theorem~\ref{the:fibrations} are satisfied   and 
all the assertions~\eqref{the:fibrations:limit_of_mg},~%
\eqref{the:fibrations:limit_of_torsion_in_homology},~%
\eqref{the:fibrations:L2-Betti},~%
\eqref{the:fibrations:limit_of_torsion}
and~\eqref{the:fibrations:L2-torsion} hold.
This follows from Corollary~\ref{cor:connected_Lie_groups_as_fibers} and the long exact homotopy
associated to the fibration above.
\end{example}

\subsection{$S^1$-actions}
\label{subsec:S1-action}

\begin{corollary}[$S^1$-action] \label{cor:S1_action}
Let $X$ be a connected finite $S^1$-$CW$-complex. Suppose that for one
(and hence all) base points $x \in X$ the composite
$\pi_1(S^1) \xrightarrow{\pi_1(\ev_x)} \pi_1(X) \xrightarrow{\phi} G$
has infinite image, where the evaluation map 
$\ev_x \colon S^1 \to X$ sends $z$ to $zx$.
(This implies that the $S^1$-action has no fixed points.) Then
all the assertions~\eqref{the:fibrations:limit_of_mg},~%
\eqref{the:fibrations:limit_of_torsion_in_homology},~%
\eqref{the:fibrations:L2-Betti},~%
\eqref{the:fibrations:limit_of_torsion}
and~\eqref{the:fibrations:L2-torsion} appearing in
Theorem~\ref{the:fibrations} hold. 
\end{corollary}

Corollary~\ref{cor:S1_action} follows directly
from Corollary~\ref{cor:connected_Lie_groups_as_fibers} provided that the action is free
and hence we obtain a fibration $S^1 \xrightarrow{\ev_x} X \to S^1\backslash X$
over the finite $CW$-complex $S^1\backslash X$. We omit the proof
of the general case, where one only have to modify 
Lemma~\ref{lem:fibrations} a little bit, namely, one uses induction
over the skeletons of the $S^1$-$CW$-complex $X$. 

\begin{corollary}[$S^1$-action on an aspherical closed  manifold]
\label{cor:S1_action_on_a_aspherical_manifold}
Let $M$ be an aspherical closed manifold with an $S^1$-action
which is non-trivial, i.e., there exists $x \in M$ and $z \in S^1$ with $zx \not= x$.
Now take $G = \pi_1(X)$ and $\phi = \id_{\pi_1(X)}$.

Then all the assertions~\eqref{the:fibrations:limit_of_mg},~%
\eqref{the:fibrations:limit_of_torsion_in_homology},~%
\eqref{the:fibrations:L2-Betti},~%
\eqref{the:fibrations:limit_of_torsion}
and~\eqref{the:fibrations:L2-torsion} appearing in
Theorem~\ref{the:fibrations} hold for this choice of $\phi$.
\end{corollary}
\begin{proof}
The $\ev_x \colon S^1 \to X$ induces an injection on the fundamental
groups~\cite[Theorem~1.43 on page~48]{Lueck(2002)}.
Now apply Corollary~\ref{cor:S1_action}.
\end{proof}

\subsection{Approximation Conjectures}
\label{subsec:Approximation_Conjectures}

The following two conjectures are motivated 
by~\cite[Conjecture~1.3]{Bergeron-Venkatesh(2010)}
and~\cite[Conjecture~11.3 on page~418 and Question~13.52 on page~478]{Lueck(2002)}. 
We will prove them in a special case
in Corollary~\ref{cor:Groups_containing_a_normal_infinite_nice_subgroups}.

\begin{conjecture}[Approximation Conjecture for $L^2$-torsion]
\label{con:Approximation_Conjecture_for_Fulgede-Kadison_determinants_and_L2-torsion}
Let $X$ be a finite connected $CW$-complex and let $\overline{X} \to X$ be a $G$-covering. 

\begin{enumerate}

\item \label{con:Approximation_Conjecture_for_Fulgede-Kadison_determinants_and_L2-torsion:torsion}
If the $G$-$CW$-structure on $\overline{X}$ and for each $i \in I$  the $CW$-structure on 
$G_i\backslash \overline{X}$  come from a given $CW$-structure on $X$, then
\[
\rho^{(2)}(\overline{X};\caln(G)) = \lim_{i \to \infty} \frac{\rho^{(2)}(G_i\backslash \overline{X};\caln(\{1\}))}{[G:G_i]};
\]

\item \label{con:Approximation_Conjecture_for_Fulgede-Kadison_determinants_and_L2-torsion:analytic}
If $X$ is a closed Riemannian manifold and we equip $G_i\backslash \overline{X}$ and $\overline{X}$
with the induced Riemannian metrics, one can replace the torsion in the equality appearing 
in~\eqref{con:Approximation_Conjecture_for_Fulgede-Kadison_determinants_and_L2-torsion:torsion}
by the analytic versions;

\item \label{con:Approximation_Conjecture_for_Fulgede-Kadison_determinants_and_L2-torsion:homology}
If $b_n^{(2)}(\overline{X};\caln(G))$ vanishes for all $n \ge 0$, then 
\[
\rho^{(2)}(\overline{X};\caln(G)) = \lim_{i \to \infty} \frac{\rho^{\IZ}(G_i\backslash \overline{X})}{[G:G_i]}.
\]
\end{enumerate}
\end{conjecture}

\begin{conjecture}[Homological growth and $L^2$-torsion for aspherical closed manifolds]
\label{con:Homological_growth_and_L2-torsion_for_aspherical_manifolds}
Let $M$ be an aspherical closed manifold of dimension $d$
and fundamental group $G = \pi_1(M)$. Then

\begin{enumerate}

\item \label{con:Homological_growth_and_L2-torsion_for_aspherical_manifolds:Betti}
For any natural number $n$ with $2n \not= d$ we have 
\[
b_n^{(2)}(\widetilde{M};\caln(G))  
=  
\lim_{i \to \infty} \frac{b_n(G_i\backslash \widetilde{M};\IQ)}{[G:G_i]} 
 =   0.
\]
If $d = 2n$ is even, we get
\[
b_n^{(2)}(\widetilde{M};\caln(G))  
=  
\lim_{i \to \infty} \frac{b_n(G_i\backslash \widetilde{M};\IQ)}{[G:G_i]} 
= 
(-1)^n \cdot \chi(M) \ge 0;
\]

\item \label{con:Homological_growth_and_L2-torsion_for_aspherical_manifolds:tors}
 For any natural number $n$ with $2n +1 \not= d$ we have 
\[
\lim_{i \in I} \;\frac{\ln\big(\bigl|\tors\bigl(H_n(G_i\backslash \widetilde{M})\bigr)\bigr|\bigr)}{[G:G_i]}
 =  0.
\]
If $d = 2n+1$, we have
\[
\lim_{i \in I} \;\frac{\ln\big(\bigl|\tors\bigl(H_n(G_i\backslash \widetilde{M})\bigr)\bigr|\bigr)}{[G:G_i]}  
=  
(-1)^n \cdot \rho^{(2)}\bigl(\widetilde{M};\caln(G)\bigr) 
\ge 0.
\]
\end{enumerate}
\end{conjecture}

For a brief survey on elementary amenable groups we refer for instance 
to~\cite[Section 6.4.1 on page 256ff]{Lueck(2002)}. Solvable groups are
examples of elementary amenable groups. Every elementary amenable group is amenable,
the converse is not true in general. 

Some evidence for the two conjectures above comes from

\begin{corollary}
\label{cor:Groups_containing_a_normal_infinite_nice_subgroups}
Let $M$ be an aspherical closed manifold with fundamental group $G = \pi_1(M)$.
Suppose that $M$ carries a non-trivial $S^1$-action or suppose
that $G$ contains a non-trivial  elementary amenable normal subgroup. Then we get for all $n \ge 0$
\begin{eqnarray*}
\lim_{i \to \infty} \frac{b_n(G_i\backslash \widetilde{M};K)}{[G:G_i]}  
& = & 
0;
\\
\lim_{i \in I} \;\frac{\mg\bigl(H_n(G_i\backslash \widetilde{M})\bigr)}{[G:G_i]}
& = & 
0;
\\
\lim_{i \in I} \;\frac{\ln\big(\bigl|\tors\bigl(H_n(G_i\backslash \widetilde{M})\bigr)\bigr|\bigr)}{[G:G_i]}
& = & 
0;
\\
\lim_{i\in I} \frac{\rho^{(2)}\bigl(G_i\backslash \widetilde{M};\caln(\{1\})\bigr)}{[G:G_i]}
& = & 
0;
\\
\lim_{i\in I} \frac{\rho^{\IZ}\bigl(G_i\backslash \widetilde{M}\bigr)}{[G:G_i]}
& = & 
0;
\\
b_n^{(2)}(\widetilde{M};\caln(G)) 
& = & 
0;
\\
\rho^{(2)}(\widetilde{M};\caln(G)) 
& = & 
0.
\end{eqnarray*}
In particular Conjecture~\ref{con:Approximation_Conjecture_for_Fulgede-Kadison_determinants_and_L2-torsion}
and Conjecture~\ref{con:Homological_growth_and_L2-torsion_for_aspherical_manifolds} are true.
\end{corollary}
\begin{proof}
The case of a non-trivial $S^1$-action has already been taken care of
in Corollary~\ref{cor:S1_action_on_a_aspherical_manifold}. 

Now suppose that  $G$ contains a non-trivial  elementary amenable normal subgroup.
We conclude from~\cite[Corollary 2 on page 240]{Hillman-Linnell(1992)}
that there is a non-trivial  normal abelian subgroup $A$ of $G$.
Since $G$ and hence $A$ are torsionfree, $A$ is infinite.  We obtain an exact
sequence of groups $1 \to A \to G \to Q$ for $Q = G/A$. Since $M$ is aspherical and hence
a model for $BG$, we obtain a fibration $BA\to M \to BQ$.  We conclude
$\pi_1(BA) = G_1(BA)$ from 
Lemma~\ref{lem:characterizing_G_1(F)}~\eqref{lem:characterizing_G_1(F):aspherical}. Hence we can apply
Theorem~\ref{the:fibrations} in the case $\phi = \id_G$ and therefore all claims follow except
$\rho^{(2)}(\widetilde{M};\caln(G)) = 0$ (since we do not know whether $BQ$ is finite.) But
$\rho^{(2)}(\widetilde{M};\caln(G)) = 0$ has already been proved in~\cite{Wegner(2009)},
see also~\cite{Li-Thom(2012)}.

\end{proof}

\subsection{Homological growth of groups}
\label{subsec:Homological_growth_of_groups}

For a group $G$ we write $H_n(G) := H_n(BG)$ and $b_n(G;K) := b_n(BG;K)$ for the classifying space $BG$. 

\begin{theorem}\label{the:growth_for_groups}
Consider a natural number $d$ and a residually finite group $G$ such that there is a model for
$BG$ with finite $(d+1)$-skeleton.  Assume either that 
$G$ contains a normal infinite  solvable subgroup or that $G$ is virtually torsionfree with finite virtual 
cohomological dimension and contains a normal infinite elementary amenable subgroup.
Then we get for all $n \le d$
\begin{eqnarray*}
\lim_{i\in I} \frac{\mg(H_n(G_i))}{[G:G_i]}  & = & 0;
\\
\lim_{i\in I} \frac{\ln\bigl(\bigl|\tors(H_n(G_i))\bigr|\bigr)}{[G:G_i]} & = & 0;
\\
\lim_{i\in I} \frac{b_n(G_i;K)}{[G:G_i]}  &= & 0.
\end{eqnarray*}
\end{theorem}
\begin{proof} Let $S \subseteq G$ be an infinite  normal solvable subgroup. 
The commutator subgroup of a group is a characteristic subgroup. Hence $S$ 
contains an infinite characteristic subgroup $S'$ whose commutator $[S',S']$ is finite. 
The subgroup $S'$ is normal in $G$.
Since $G$ is residually finite and $[S',S']$ is finite, we can find a normal subgroup
$H' \subseteq G$ of finite index with $H' \cap [S',S'] = \{1\}$. Let $H := \bigcap_{\mu \in \aut(G)} \mu(H')$.
Obviously $H$ is a characteristic subgroup of $G$.
Since $G$ is finitely generated and hence contains only finitely many normal subgroup of index $[G:H']$,
$H$ has finite index in $G$. The intersection $A := H \cap S'$ is an infinite normal subgroup of $G$. 
Since $H \cap [S',S'] = \{1\}$, the projection $S'\to S'/[S',S']$ is injective on $A$.
Hence $A$ is an infinite abelian normal subgroup of $G$.

Put $Q = G/A$. We obtain a fibration of connected $CW$-complexes
$BA \to BG \to BQ$. We conclude $G_1(BA) = \pi_1(BA)$ from 
Lemma~\ref{lem:characterizing_G_1(F)}~\eqref{lem:characterizing_G_1(F):aspherical}.
Now we can apply Theorem~\ref{the:fibrations} to the fibration
$BA \to BG \to BQ$ for the canonical isomorphism $\phi \colon \pi_1(BG) \to G$.

Let $G' \subseteq G$ be a torsionfree subgroup of finite index
whose cohomological dimension is finite. There is an index $i_0 \in I$ such that
$G_i \subseteq G'$ holds for $i \ge i_0$. Put $G_i' := G_i \cap G'$. Since
$[G:G_i] = [G:G'] \cdot [G':G_i']$ holds for $i \ge i_0$, it suffices to prove the claim
for $G'$ and the system $\{G_i' \mid i \in I\}$.  The group $G'$ contains a  normal 
infinite abelian subgroup $H' \subseteq G'$
by~\cite[Corollary 2 on page 240]{Hillman-Linnell(1992)}.
Now the claim follows for $G'$ and hence for $G$ from the argument above.
\end{proof}

\subsection{Rank gradient and cost}
\label{subsec:Exponent_gradient}

Let $G$ be a finitely generated group. Let $H \subseteq G$ be a subgroup of finite index $[G:H]$. Then 
$\frac{\mg(H) -1}{[G:H]} \le \mg(G)-1$. Therefore the following limit exists and is called
\emph{rank gradient} (see Lackenby~\cite{Lackenby(2005expanders)})
\begin{eqnarray}
RG(G;(G_i)_{i\in I}) &:= & \lim_{i \in I} \frac{\mg(G_i) -1}{[G:G_i]}.
\label{rank_gradient}
\end{eqnarray}
It is not known in general whether this limit is independent of the choice of the system $(G_i)_{i\in I}$.

We have (cf.  Lemma~\ref{lem:properties_of_mg}~\eqref{lem:properties_of_mg:estimate})
\begin{eqnarray*}
   b_1(BG_i;K) 
   & \le & 
   \mg(G_i);
   \\
    \mg(H_1(BG_i)) 
    & \le & 
    b_1(BG_i;\IQ) + \frac{\ln\bigl(|\tors(H_1(BG_i))|\bigr)}{\ln(2)}.
\end{eqnarray*}

\begin{question}\label{que:RG_and_limits_of_Betti_numbers}
For which finitely generated groups $G$, sequences $(G_i)_{i \in I}$ with $\bigcap_{i \in I} G_i = \{1\}$ and fields $K$, do  we have
\[RG(G;(G_i)_{i\in I}) =  \limsup_{i \in I} \frac{b_1(G_i;K)}{[G:G_i]}?\]
\end{question}
\begin{question}\label{que:limit_independent_of_K}
For which groups $G$,  does the limit $\lim_{i \in I} \frac{b_1(G_i;K)}{[G:G_i]}$ exist for
all systems $(G_i)_{i\in I}$ with $\bigcap_{i \in I} G_i = \{1\}$ and fields $K$ and is independent of the choice of 
$(G_i)_{i\in I}$ and $K$?
\end{question}

Ab\'ert-Nikolov~\cite[Theorem~3]{Abert-Nikolov(2012)}
have shown for a finitely presented residually finite group $G$ which contains a normal infinite amenable subgroup
that $RG(G;(G_i)_{i\in I}) =  0$ holds for all systems $(G_i)_{i\in I}$ (with trivial intersection).
Hence the answer to the Questions~\ref{que:RG_and_limits_of_Betti_numbers} 
and~\ref{que:limit_independent_of_K} above is yes for such groups. 

The questions above is related to questions of 
Gaboriau (see~\cite{Gaboriau(2000b),Gaboriau(2002a),Gaboriau(2002b)}),
whether every essentially free measure preserving Borel action of a group
has the same cost,  and whether the difference of the cost and the first $L^2$-Betti number of a measurable
equivalence relation is always equal to $1$.

The answer to the Questions~\ref{que:RG_and_limits_of_Betti_numbers} 
and~\ref{que:limit_independent_of_K} is negative in general if we drop the condition
that the system $\{G_i\mid i \in I\}$ has non-trivial
intersection, as the following example shows.

\begin{example} Consider a group $H$. Put $G = \IZ \ast H$.
For a natural number $n$ let $G_n \subseteq G$ be the preimage of $n
\cdot \IZ$ under the projection $\pr \colon G= \IZ \ast H \to \IZ$. We obtain
a inverse system of subgroups $\{G_i \mid i = 1,2,3, \ldots \}$ directed by the property $i_1$ divides $i_2$ satisfying 
$\bigcap_{i \ge 1} G_i = \ker(\pr)$.
Let $BG_i \to BG$ be the covering of $BG$ associated
to $G_i \subseteq G$. Then $BG_i$ is homeomorphic to $S^1 \vee \bigvee_{j= 1}^i BH$.
We have
\[G_i \cong \pi_1(BG_i) \cong \pi_1(S^1 \vee \bigvee_{j = 1}^i BH) \cong \IZ \ast \ast_{j=1}^i H.\]
Since $\mg(A \ast B) = \mg(A ) + \mg(B)$ holds (see~\cite[Corollary~2 in Section~8.5 on page~227]{Cohen(1989)}, we conclude
\begin{eqnarray*}
H_1(G_i;K) & = & K \oplus \bigoplus_{j=1}^i H_1(H;K);
\\
H_1(G_i) & = & \IZ\oplus \bigoplus_{j=1}^i H_1(H);
\\
\mg(G_i) & = & 1 + i \cdot \mg(H);
\\
\lim_{i \to \infty} \frac{b_1(G_i;K)}{i} & = & b_1(H;K);
\\
\lim_{i \to \infty} \frac{\mg(H_1(G_i))}{i} & = & \mg(H_1(H));
\\
R(G;(G_i)_{i \ge 1}) & = & \mg(H).
\end{eqnarray*}
Let $q$ be a prime different from $p$. Put $H = \IZ/p \ast \IZ/q \ast \IZ/q$.
Then $b_1(H;\IQ) = 0$, $b_1(H;\IF_p) = 1$, $\mg(H_1(H)) = 2$, and $\mg(H) = 3$.
Hence we obtain
\[\lim_{i \to \infty} \frac{b_1(G_i;\IQ )}{i} < \lim_{i \to \infty} \frac{b_1(G_i;\IF_p)}{i} 
< \lim_{i \to \infty} \frac{\mg(H_1(G_i))}{i} < R(G;(G_i)_{i \ge 1}).\]

Obviously $BH$ can be chosen to be of finite type. Let $BH^{(2)}$ be its
two-skeleton. Put $X = S^1 \vee BH^{(2)}$.  This is a finite $2$-dimensional
$CW$-complex.  By Maunder's short proof of the Kan-Thurston Theorem
(see~\cite{Maunder(1981)}), we can find a group $\Gamma$ with a finite
$2$-dimensional model for $B\Gamma$ together with a map $f \colon B\Gamma \to X$
such that for any local coefficient system $M$ of $X$ the map $f$ induces an
isomorphism $f_* \colon H_n(B\Gamma;f^*M) \xrightarrow{\cong} H_n(X;M)$ and the
map $\pi_1(f) \colon \pi_1(B\Gamma) = \Gamma \to G = \pi_1(X)$ is surjective. If
$Z$ is a connected $CW$-complex, $H \subseteq \pi_1(Z)$ is a subgroup,
$K[\pi_1(Z)/H]$ is the obvious coefficient system on $Z$, and $\overline{Z} \to
Z$ is the $H$-covering associated to the subgroup $H \subseteq \pi_1(Z)$, then
$H_n(\overline{Z};K) \cong H_n(Z;K[\pi_1(Z)/H])$.  Let $\Gamma_i \subseteq
\Gamma$ be the preimage of $G_i$ under the epimorphism $\pi_1(f) \colon
\pi_1(B\Gamma) = \Gamma \to G := \pi_1(X)$.  We conclude
\begin{eqnarray*}
b_1(\Gamma_i;K) & = & b_1(G_i;K);
\\
\mg(H_1(\Gamma_i)) & = & \mg(H_1(G_i)); 
\\
\mg(\Gamma_i) & \ge & \mg(G_i).
\end{eqnarray*}
Hence we get
\[\lim_{i \to \infty} \frac{b_1(\Gamma_i;\IQ )}{i} < \lim_{i \to \infty} \frac{b_1(\Gamma_i;\IF_p)}{i} 
< \lim_{i \to \infty} \frac{\mg(H_1(G_i))}{i} < R(\Gamma;(\Gamma_i)_{i \ge 1}).\]
The advantage of the more elaborate construction using Maunder's result  is
that there is a finite $2$-dimensional model for $B\Gamma$ and $\Gamma$ is in particular torsionfree.
On the other hand we have no idea what the group $\Gamma$ is.

Other examples of this kind can be found in~\cite{Bergeron-Linnell-Lueck-Sauer(2012),Ershof-Lueck(2012)}.
\end{example}


 \typeout{------------------------   Section 2: Preliminaries --------------------}

\section{Preliminaries}
\label{sec:Preliminaries}

In this section we present some preliminaries for the proof of
our main Theorem~\ref{the:fibrations}.

\subsection{Gottlieb's subgroup of the fundamental
  group}
\label{subsec:Gottlieb's_subgroup_of_the_fundamental_group}

We have defined $G_1(F,x_0)   \subseteq \pi_1(F,x_0)$ in Definition~\ref{def:G_1(F)}. We collect some basic properties.

If $v$ is a path in $F$ from $x_0$ to $x_1$, then the associated
isomorphism $c_w \colon \pi_1(F,x) \to \pi_1(F,y)$ given by 
$[w] \mapsto [v^- \ast w \ast v]$ induces an isomorphism 
$c_w' \colon G_1(F,x) \xrightarrow{\cong} G_1(F,y)$ which is the identity in the case 
$x = y$. Therefore we can and will suppress the base point $x_0 \in F$.

The subgroup $G_1(F)$ was originally defined by
Gottlieb~\cite{Gottlieb(1965)}. The elementary proof of the next lemma can be
found in~\cite{Gottlieb(1965)} and~\cite[Proposition~4.3]{Lueck(1987)}.

\begin{lemma}\label{lem:characterizing_G_1(F)} Let $F$ be a connected $CW$-complex. Then:
  \begin{enumerate}

  \item \label{lem:characterizing_G_1(F):universal_covering}
   An element $g \in \pi_1(F)$ belongs to $G_1(F)$ if and only if it belongs to
  the center of $\pi_1(F)$ and the map $l_g \colon \widetilde{F} \to \widetilde{F}$ on
  the universal covering given by multiplication with $g$ is
  $\pi_1(F)$-homotopic to the identity;

\item \label{lem:characterizing_G_1(F):aspherical} $G_1(F)$ is contained in the
  center of $\pi_1(F)$. If $F$ is aspherical, $G_1(F)$ agrees with the center of
  $\pi_1(F)$;

\item \label{lem:characterizing_G_1(F):example} If $F$ is a connected Lie group,
  or more generally a connected H-space, then $G_1(F) = \pi_1(F)$;

\item \label{lem:characterizing_G_1(F):products} If $F'$ is another connected
  $CW$-complex, then we obtain a canonical isomorphism
  \[G_1(F \times F') \xrightarrow{\cong} G_1(F) \times G_1(F');\]

  \item \label{lem:characterizing_G_1(F):homotopy_invariance} 
  If $f \colon F \to F'$ is a pointed homotopy equivalence
  of connected $CW$-complexes, then it induces an isomorphism
   \[
   G_1(F) \xrightarrow{\cong} G_1(F');
   \]

   \item \label{lem:characterizing_G_1(F):S1-splitting}  $F$ is homotopy equivalent to 
   $S^1 \times F'$ for some $CW$-complex $F'$ if and only there exists an isomorphism
   $f \colon \IZ \times H\to \pi_1(F)$ for some group $H$ such that $f(\IZ) \subseteq G_1(F)$.

\end{enumerate}
\end{lemma}

\subsection{Relating $L^2$-torsion and $\IZ$-torsion}
\label{subsec:Relating_L2-torsion_and_Z-torsion}

Let $C_*$ be a finite based free $\IZ$-chain complex, for instance the cellular
chain complex $C_*(X) $ of a finite $CW$-complex. Let $H_n^{(2)}(C_*^{(2)})$ be
the $L^2$-homology of $C_*^{(2)}$ with respect to the von Neumann algebra
$\caln(\{1\}) = \IC$. The underlying complex vector space is the homology
$H_n(\IC \otimes_{\IC} C_*)$ of $\IC \otimes_{\IZ} C_*(X)$ but it comes now with
the structure of a Hilbert space.  For the reader's convenience we recall this
Hilbert space structure. Recall that $\IC \otimes_{\IZ} C_n(X)$ inherits from
the cellular $\IZ$-basis on $C_n(X)$ and the standard Hilbert space structure on
$\IC$ the structure of a Hilbert space and the resulting $L^2$-chain complex is
denoted by $C_*^{(2)}$.  Let
\[\Delta_n^{(2)} = \bigl(c_n^{(2)}\bigr)^*\circ c_n^{(2)} + c_{n+1}^{(2)} \circ
\bigl(c_{n+1}^{(2)}\bigr)^* \colon C_n^{(2)} \to C_n^{(2)}\] be the associated
Laplacian.  Equip $\ker(\Delta_n^{(2)}) \subseteq C_n^{(2)}$ with the induced
Hilbert space structure.  Equip $H_n^{(2)}(C_*^{(2)})$ with the Hilbert space
structure for which the obvious $\IC$-isomorphism $\ker(\Delta_n^{(2)}) \to
H_n^{(2)}(C_*^{(2)})$ becomes an isometric isomorphism. This is the same as the
Hilbert quotient structure with respect to the projection
$\ker\bigl(c_n^{(2)}\bigr) \to H_n^{(2)}(C_*^{(2)})$, if we equip
$\ker\bigl(c_n^{(2)}\bigr) \subseteq C_n^{(2)}$ with the Hilbert subspace
structure.

\begin{notation} \label{not_M_f} If $M$ is a finitely generated abelian group,
  define
  \begin{eqnarray*}
    M_f & := & M/ \tors(M).
  \end{eqnarray*}
\end{notation}

Choose a $\IZ$-basis on $H_n(C_*)_f$. This and the standard Hilbert space
structure on $\IC$ induces a Hilbert space structure on $\IC\otimes_{\IZ} H_n(C_*)_f$.
We denote this Hilbert space by $\bigl(H_n(C_*)_f\bigr)^{(2)} $.
We have the canonical $\IC$-isomorphism
\begin{eqnarray}
  \alpha_n \colon \bigl(H_n(C_*)_f\bigr)^{(2)} 
  & \xrightarrow{\cong}  & 
  H_n^{(2)}(C_*^{(2)}).
  \label{iso_alpha}
\end{eqnarray}
  
Now we can consider the logarithm of the Fuglede-Kadison
determinant
\[
\ln\bigl({\det}_{\caln(\{1\})}\bigl(\alpha_n \colon \bigl(H_n(C_*)_f)^{(2)} \to
H_n^{(2)}(C_*^{(2)})\bigr)\bigr) \in \IR.
\]
It is independent of the choice of the $\IZ$-basis of $H_n(C_*)_f$, since the absolute value of
the determinant of an invertible matrix over $\IZ$ is always $1$.

\begin{lemma} \label{lem:rho(2)-rhoZ} 
  Let $C_*$ be a finite based free $\IZ$-chain complex. Then
  \begin{eqnarray*}
    \rho^{\IZ}(C_*)  - \rho^{(2)}\bigl(C_*^{(2)};\caln(\{1\})\bigr) 
    & = &
    \sum_{n \ge 0} (-1)^n \cdot 
    \ln\bigl({\det}_{\caln(\{1\})}(\alpha_n)\bigr).
  \end{eqnarray*}
\end{lemma}
\begin{proof}
  Since $\IZ$ is a principal ideal domain and $C_*$ is a finite free $\IZ$-chain
  complex, $H_n(C_*)_f$ is finitely generated free and $\tors(H_n(C_*))$ is a
  finite product of finite cyclic groups.  Denote by $n\bigl[H_n(C_*)_f\bigr]$
  the $\IZ$-chain complex which is concentrated in dimension $n$ and has as
  $n$th chain module $H_n(C_*)_f$. In the sequel we equip $H_n(C_*)_f$ with
  some $\IZ$-basis.  For every $n \ge 0$, one can find finite based free
  $\IZ$-chain complexes $F^{n,i}_*$ for $i = 1,2, \ldots, r_n$ such that
  $F^{n,i}_*$ is concentrated in dimensions $(n+1)$ and $n$, its $(n+1)$th
  differential is given by multiplication $\IZ \xrightarrow{l_{n,i}} \IZ$ with
  some integer $l_{n,i} \ge 2$ and there is an isomorphism
  \[
  \eta_n \colon H_n(C_*)_f \oplus \bigoplus_{i =1}^{r_n} H_n(F_*^{n,i})
  \xrightarrow{\cong} H_n(C_*)
  \]
  of abelian group such that the composite of the inverse of $\eta_n$ and the
  canonical projection $H_n(C_*)_f \oplus \bigoplus_{i =1}^{r_n} H_n(F_*^{p,i})
  \to H_n(C_*)_f$ agrees with the canonical projection $H_n(C_*) \to
  H_n(C_*)_f$.  Using the exact sequence of $\IZ$-modules $C_{n+1}
  \xrightarrow{c_{n+1}} \ker(c_n) \to H_n(C_*) \to 0$ one constructs a
  $\IZ$-chain map
  \[
  f_* \colon \bigoplus_{n \ge 0} \left(n\bigl[H_n(C_*)_f\bigr] \oplus 
  \bigoplus_{i=1}^{r_n} F^{n,i}_* \right) \to C_*
  \]
  such that $H_n(f_*)$ is the isomorphism $\eta_n$. Since the source and the
  target of $f_*$ are free $\IZ$-chain complexes, $f_*$ is a $\IZ$-chain
  homotopy equivalence.  The $L^2$-torsion $t^{(2)}\bigl(f_*^{(2)}\bigr)$ of
  $f_*^{(2)}$ is defined in~\cite[Definition~3.31 on page~141]{Lueck(2002)}. We
  conclude from~\cite[Theorem~3.35~(5) on page~142]{Lueck(2002)})
  \begin{multline}
    t^{(2)}\bigl(f_*^{(2)}\bigr) = \rho^{(2)}\bigl(C_*^{(2)}\bigr) -
    \rho^{(2)}\left(\bigoplus_{n  \ge 0} \left(n\bigl[H_n(C_*)_f\bigr] \oplus 
      \bigoplus_{i=1}^{r_n} F^{n,i}_* \right)\right) \\+ \sum_{n \ge 0} (-1)^n \cdot
    \ln\left(\det\bigl(H_n(f^{(2)}_*)\bigr)\right).
    \label{lem:rho(2)-rhoZ:difference_of_rho-s}
  \end{multline}
  Since $f_*$ is a $\IZ$-chain homotopy equivalence and hence its Whitehead torsion is an
  element in $\{\pm 1\}$, we get
  \begin{eqnarray}
    & t^{(2)}\bigl(f_*^{(2)}\bigr)  =  \ln(|\pm 1|) = 0. &
    \label{lem:rho(2)-rhoZ:t(2)(f_ast(2))_is_zero}
  \end{eqnarray}
  We obtain an equality of maps of finitely generated Hilbert
  $\caln(\{1\})$-modules
  \begin{eqnarray}
    H_n\bigl(f^{(2)}_*\bigr) & = &\alpha_n
    \label{lem:rho(2)-rhoZ:alpha_n}.
  \end{eqnarray}
  One easily checks
  \begin{multline}
    \quad \rho^{(2)}\left(\bigoplus_{n \ge 0} 
   \left(n\bigl[H_n(C_*)_f\bigr] ^{(2)}\oplus \bigoplus_{i=1}^{r_n} \bigl(F^{n,i}_*\bigr)^{(2)}\right)
    \right) 
    \\ 
    = \sum_{n \ge 0} \sum_{i=1}^{r_n}
    \rho^{(2)}\left(\bigl(F^{n,i}_*\bigr)^{(2)};\caln(\{1\}) \right);
    \label{lem:rho(2)-rhoZ:rho(2)-s}
  \end{multline}
  \begin{eqnarray}
    \rho^{\IZ}(C_*)
    & = & 
    \sum_{n \ge 0}
    \sum_{i=1}^{r_n} \rho^{\IZ}\bigl(F^{n,i}_*\bigr). 
    \label{lem:rho(2)-rhoZ:rho(Z)-s}
  \end{eqnarray}
  Because of~\eqref{lem:rho(2)-rhoZ:difference_of_rho-s},
  \eqref{lem:rho(2)-rhoZ:t(2)(f_ast(2))_is_zero},~%
  \eqref{lem:rho(2)-rhoZ:alpha_n},~\eqref{lem:rho(2)-rhoZ:rho(2)-s}
  and~\eqref{lem:rho(2)-rhoZ:rho(Z)-s} it suffices to show
  \begin{eqnarray*}
    \rho^{(2)}\left(\bigl(F^{n,i}_*\bigr)^{(2)}\right)  
    & = & 
    \rho^{\IZ}\bigl(F^{n,i}_*\bigr). 
    \end{eqnarray*}
  This is done by the following calculation.
  \begin{eqnarray*}
    \rho^{(2)}\bigl((F^{n,i}_*)^{(2)}\bigr) 
    & = & 
    (-1)^{n} \cdot \ln\bigl({\det}_{\caln(\{1\})}\bigl(l_{n,i} \colon \IC \to \IC\bigr)\bigr) 
    \\
    & = & 
    (-1)^{n} \cdot \ln(l_{i,n}) 
    \\
    & = &
    (-1)^{n} \cdot \ln\left(\bigl|\tors\bigl(H_n(F^{n,i}_*)\bigr)\bigr|\right)
    \\
    & = & 
    \rho^{\IZ}\bigl(F^{n,i}_*\bigr).
  \end{eqnarray*}
  This finishes the proof of Lemma~\ref{lem:rho(2)-rhoZ}.
\end{proof}

\begin{notation}\label{not:closure_of_submodule}
  Let $A$ be a finitely generated free abelian group and let $B \subseteq A$ be a
  subgroup.  Define the \emph{closure} of $B$ in $A$ to be the subgroup
  \begin{eqnarray*}
    \overline{B}  
    & = & \{x \in A \mid n\cdot x \in B \; \text{for some non-zero integer}\; n\}.
  \end{eqnarray*}
\end{notation}

\begin{lemma} \label{lem:Fuglede-Kadison_and_tors_for_G_is_trivial} Let $u
  \colon \IZ^r \to \IZ^s$ be a homomorphism of abelian groups. Let $j_k \colon
  \ker(u) \to \IZ^r$ be the inclusion and $\pr_c \colon \IZ^s \to \coker(u)_f$
  be the canonical projection.  Choose $\IZ$-basis for $\ker(u)$ and
  $\coker(u)_f$.

  Then ${\det}_{\caln(\{1\})}\bigl(j_k^{(2)}\bigr)$ and
  ${\det}_{\caln(\{1\})}\bigl(\pr_c^{(2)}\bigr)$ are independent of the choice
  of the $\IZ$-basis for $\ker(u)$ and $\coker(u)_f$, and we have
  \begin{eqnarray*}
    {\det}_{\caln(\{1\})}(u^{(2)}) & = &
    {\det}_{\caln(\{1\})}\bigl(j_k^{(2)}\bigr) \cdot
    \bigl|\tors(\coker(u))\bigr|
    \cdot {\det}_{\caln(\{1\})}\bigl(\pr_c^{(2)}\bigr),
  \end{eqnarray*}
  and
  \[
  \begin{array}{lclcl}
    1 & \le & {\det}_{\caln(\{1\})}(j_k^{(2)}) & \le & {\det}_{\caln(\{1\})}(u^{(2)});
    \\
    1 & \le &  {\det}_{\caln(\{1\})}\bigl(\pr_c^{(2)}\bigr) 
    & \le & {\det}_{\caln(\{1\})}(u^{(2)});
    \\
    1 & \le & \bigl|\tors(\coker(u))\bigr| & \le & {\det}_{\caln(\{1\})}(u^{(2)}).
  \end{array}
  \]
\end{lemma}
\begin{proof}
  Because of~\cite[Lemma~13.12 on
  page~459]{Lueck(2002)} the trivial group satisfies the Determinant
  Conjecture. This implies $1 \le {\det}_{\caln(\{1\})}(u^{(2)})$
  and $1 \le {\det}_{\caln(\{1\})}\bigl(\pr_c^{(2)}\bigr)$ and that for any
  isomorphism $j \colon \IZ^l \to \IZ^l$ we have $\det_{\caln(\{1\})}(j^{(2)}) =
  1$. The latter together with~\cite[Theorem~3.14~(1) on page~128]{Lueck(2002)}
  implies that the $\IZ$-basis for $\ker(u)$ and $\coker(u)_f$ do not matter and
  that it suffices to show the equation
  \begin{eqnarray*}
    {\det}_{\caln(\{1\})}(u^{(2)}) & = &
    {\det}_{\caln(\{1\})}\bigl(j_k^{(2)}\bigr) \cdot
    \bigl|\tors(\coker(u))\bigr|
    \cdot {\det}_{\caln(\{1\})}\bigl(\pr_c^{(2)}\bigr).
  \end{eqnarray*}

  Let $\overline{u}$ be the map of finitely generated free abelian groups
  $\IZ^r/\ker(f) \to \overline{\im(u)}$ induced by $u$.  Equip $\IZ^r/\ker(f)$
  and $\overline{\im(u)}$ with $\IZ$-basis.  Let $\pr_k \colon \IZ^r \to
  \IZ^r/\ker(u)$ be the canonical projection and $j_i \colon \overline{\im(u)}
  \to \IZ^s$ be the inclusion.  We conclude from~\cite[Theorem~3.14~(1) on
  page~128]{Lueck(2002)}
  \[ {\det}_{\caln(\{1\})}(u^{(2)}) = {\det}_{\caln(\{1\})}\bigl(j_i^{(2)}\bigr)
  \cdot {\det}_{\caln(\{1\})}\bigl(\overline{u}^{(2)}\bigr) \cdot
  {\det}_{\caln(\{1\})}\bigl(\pr_k^{(2)}\bigr).
  \]
  Hence it remains to show
  \begin{eqnarray*} {\det}_{\caln(\{1\})}\bigl(j_i^{(2)}\bigr) 
    & = &
    {\det}_{\caln(\{1\})}\bigl(\pr_c^{(2)}\bigr);
    \\
    {\det}_{\caln(\{1\})}\bigl(\pr_k^{(2)}\bigr) 
    & = &
    {\det}_{\caln(\{1\})}\bigl(j_k^{(2)}\bigr).
  \end{eqnarray*}
  This follows from Lemma~\ref{lem:rho(2)-rhoZ} applied to the $2$-dimensional
  acyclic based free finite $\IZ$-chain complexes $ \ker(u) \xrightarrow{j_k}
  \IZ^r \xrightarrow{\pr_k} \IZ^r/\ker(u) $ and $ \overline{\im(u)}
  \xrightarrow{j_i} \IZ^1 \xrightarrow{\pr_c} \coker(u)_f.  $ This finishes the
  proof of Lemma~\ref{lem:Fuglede-Kadison_and_tors_for_G_is_trivial}.
\end{proof}

\subsection{Minimal numbers of generators}
\label{subsec:Minimal_numbers_of_generators}

\begin{definition}[Minimal numbers of generators]
  \label{def:minimal_numbers_of_generators}
  Let $G$ be a finitely generated group. Denote by $\mg(G) \in \IN$ the
  \emph{minimal numbers of generators} of $G$, where we put $\mg(\{1\}) = 0$.
\end{definition}

If $M$ is an abelian group, then $\mg(M)$ is the minimum over all natural
numbers $n$ for which there exists an epimorphism of abelian groups $\IZ^n \to
M$.

\begin{lemma} \label{lem:properties_of_mg}

  \begin{enumerate}

  \item \label{lem:properties_of_mg:structure_d} Let $r \ge 0$, $s \ge 0$ and
    $d_1 \ge 2$, $d_2 \ge 2, \ldots, d_s \ge 2$ be integers with $d_1 | d_2 |
    d_3 | \cdots | d_s$ and $M \cong \IZ^r \oplus \bigoplus_{i = 1}^s
    \IZ/d_i$. Then
    \[
    \mg(M) = r + s;
    \]

  \item \label{lem:properties_of_mg:structure_nrimes} Let $M$ be a finitely
    generated abelian group. For any prime $p$ we can write the $p$-Sylow
    subgroup $\tors(M)_p$ of the finite abelian group $\tors(M)$ as a direct sum
    \[\tors(M)_p \cong \bigoplus_{i = 1}^{s_p} \IZ/p^{k_i}\]
    for integers $k_i \ge 1$ and $s_p \ge 0$. Then
    \[
    \mg(M) = \dim_{\IQ}(\IQ \otimes_{\IZ} M)+ \max\{s_p \mid p \;
    \text{prime}\}.
    \]
    In particular
    \[
    \mg(M) = \mg(M/\tors(M)) + \mg(\tors(M));
    \]

  \item \label{lem:properties_of_mg:estimate} If $K$ is any field, we have
    \[
    \dim_{\IQ}(\IQ\otimes_{\IZ} M) \le \dim_{K}(K \otimes_{\IZ} M) \le \mg(M) \le \dim_{\IQ}(\IQ\otimes_{\IZ} M) +
    \frac{\ln\bigl(|\tors(M)|\bigr)}{\ln(2)};
  \]

  \item \label{lem:properties_of_mg:subadditivity} If $0 \to M_0 \to M_1 \to M_2
    \to 0$ is an exact sequence of finitely generated $\IZ$-modules, then
    \[
    \mg(M_1) \le \mg(M_0) + \mg(M_2).
    \]
    If $M$ is a submodule of the finitely generated $\IZ$-module $M$, then
    \[
    \mg(M) \le \mg(N);
    \]
    If $N$ is a quotient module of the finitely generated $\IZ$-module $M$, then
    \[
    \mg(N) \le \mg(M).
    \]

  \end{enumerate}
\end{lemma}
\begin{proof}~\eqref{lem:properties_of_mg:structure_d} Consider any epimorphism
  $f \colon \IZ^n \to M$. We have to show $n \ge r+s$.  The kernel of $f$ is a
  finitely generated free $\IZ$-module of rank $m \le n$.  Hence we obtain an
  exact sequence $0 \to \IZ^m \xrightarrow{j} \IZ^n \xrightarrow{f} M \to 0$.
  By composing $i$ with appropriate $\IZ$-isomorphisms from the left and the
  right we can arrange that there are integers $c_1 \ge 1$, $c_2 \ge 1, \ldots,
  c_m \ge 1$ such that $c_1 | c_2 | c_3 | \ldots | c_m$, and $j$ sends the
  element $e_i$, whose entries are all zero except the $i$th entry which is
  $1$, to $c_i \cdot e_i$ for $i = 1,2 \ldots, m$. Let $t \in \{0,1,2, \ldots , m\}$ 
  be the integer for which $c_i = 1$ for $i \le t$. Then
  \[
  M \cong \IZ^{n-m} \oplus \bigoplus_{i = t+1} ^m \IZ/c_i.
  \]
  From the structure theorem for abelian groups we conclude $n-m= r$ and 
  $m-t = s$. This implies
  \[
  r + s = n-m + m- t \le n.
  \]%
  \eqref{lem:properties_of_mg:structure_nrimes}
  and~\eqref{lem:properties_of_mg:estimate} These follow from
  assertion~\eqref{lem:properties_of_mg:structure_d}.
  \\[1mm]~\eqref{lem:properties_of_mg:subadditivity} Let $0 \to M_0
  \xrightarrow{j} M_1 \xrightarrow{q} M_2 \to 0$ be an exact sequence of
  finitely generated $\IZ$-modules. Choose epimorphisms $f_i \colon
  \IZ^{\mg(M_i)} \to M_i$ for $i = 0,2$.  Choose a homomorphism $\widetilde{f_2}
  \colon \IZ^{\mg(M_2)} \to M_2$ with $q \circ \widetilde{f_2} = f_2$.  Define
  the $\IZ$-homomorphism $f_1 \colon \IZ^{\mg(M_0) + \mg(M_2)} = \IZ^{\mg(M_0)}
  \oplus \IZ^{\mg(M_2)} \to M_1$ by $(j \circ f_0) \oplus \widetilde{f_2}$. One
  easily checks that $f_1$ is surjective. Hence
  \[\mg(M_1) \le \mg(M_0) + \mg(M_2).
  \]
  Let $M$ be a submodule of the finitely generated $\IZ$-module $N$. Choose an
  epimorphism $f \colon \IZ^{\mg(N)} \to N$. Then $f$ induces an epimorphism
  $f^{-1}(M) \to M$ and $f^{-1}(M)$ is isomorphic to $\IZ^m$ for some integer $m
  \le \mg(N)$. Hence $\mg(M) \le m \le \mg(N)$. The claim for quotient modules
  is obvious.
\end{proof}

\begin{remark} \label{rem:mg}
Notice that the minimal number of generators  is not additive under exact sequences, actually it is not even
additive under direct sums.  Namely, if $p$ and $q$ are different prime numbers,
then $\mg(\IZ/p \oplus \IZ/q) = \mg(\IZ/p) = \mg(\IZ/q) = 1$.  Moreover, for
non-abelian groups it is not true that $\mg(H) \le \mg(G)$ holds for a subgroup
$H \subseteq G$, e.g., this fails for $G$ the free group of two generators and
$H$ any proper subgroup of finite index.
\end{remark}

\begin{remark}\label{rem:tro_more_difficult}
The main reason why Betti numbers and the number of minimal generators are easier to handle
than the order of the torsion in homology is that the minimal number of generators satisfies
assertion~\eqref{lem:properties_of_mg:subadditivity} of Lemma~\ref{lem:properties_of_mg}
and the dimension of $K$-vector spaces satisfies the analogous statement.
This is not true for the order of the torsion in a finitely generated abelian group. Passing to the quotient
of a finitely generated abelian group may increase the order of the torsion submodule
as the example $\IZ \to \IZ/n$ shows.
\end{remark}

\begin{lemma} \label{lemma:order_of_group_homology} Let $G$ be a finite
  group. Let $M$ be a $\IZ G$-module which is finitely generated as abelian
  group.  Consider $n \ge 1$. Let $d_n$ be an integer such that there is a free $\IZ
  G$-resolution $F_*$ of the trivial $\IZ G$-module $\IZ$ satisfying $\dim_{\IZ
    G}(F_n) \le d_n$. ($F_*$ may depend on $n$.)

  Then $H_n(G;M)$ is annihilated by multiplication with $|G|$ and we get
  \begin{eqnarray*}
    \mg(H_n(G;M)) & \le &  d_n \cdot \mg(M);
    \\
    |H_n(G;M)| & \le & |G|^{d_n\cdot \mg(M)}.
  \end{eqnarray*}
\end{lemma}
\begin{proof}
  Fix $n \ge 1$. Then $H_n(G;M)$ is annihilated by multiplication with $|G|$
  by~\cite[Corollary~10.2 on page 84]{Brown(1982)}.  By definition 
  $H_n(G;M) =  H_n(M \otimes_{\IZ G} F_*)$.  We conclude from
  Lemma~\ref{lem:properties_of_mg}~\eqref{lem:properties_of_mg:subadditivity}
  \begin{multline*}
    \mg\bigl(H_n(G;M)\bigr) = \mg\bigl(H_n(M \otimes_{\IZ G} F_*)\bigr) \le
    \mg(M \otimes_{\IZ G} F_n\bigr)
    \\
    = \mg\bigl(M^{\dim_{\IZ G}(F_n)}\bigr) \le \dim_{\IZ G}(F_n) \cdot \mg(M)
    \le d_n \cdot \mg(M).
  \end{multline*}
  Consider an epimorphism $f \colon \IZ^l \to H_n(G;M)$. Since multiplication
  with $|G|$ annihilates $H_n(G;M)$, it induces an epimorphism 
  $\overline{f} \colon (\IZ/|G|)^l \to H_n(G;M)$. This implies $|H_n(G;M)| \le |G|^l$. 
  We conclude
  \[
  |H_n(G;M)| \le |G|^{\mg(H_n(G;M))} \le |G|^{d_n\cdot \mg(M)}.
  \]
\end{proof}

\subsection{Nilpotent modules}
\label{subsec:nilpotent_modules}

\begin{definition}[Nilpotent $\IZ G$-module]
  We call a $\IZ G$-module $M$ \emph{nilpotent}, if there exists an integer $r
  \ge 0$ and a filtration $\{0\} = M_0 \subseteq M_1 \subseteq M_2 \subseteq
  \ldots \subseteq M_r = M$ such that the $G$-action on $M_i/M_{i-1}$ is trivial
  for $i = 1,2, \ldots, r$. The minimal number $r$ for which such a filtration
  exists is called the \emph{filtration length} of $M$.
\end{definition}

A $\IZ G$-module is trivial if and only if it is nilpotent of filtration length
$0$.  The $G$-action on an $\IZ G$-module $M$ is trivial if and only if it is
nilpotent of filtration length at most $1$. The elementary proof of the next
lemma is left to the reader.

\begin{lemma} \label{lem_filtration_length} Let $G$ be a group. Then:
  \begin{enumerate}

  \item \label{lem_filtration_length:submodule} A $\IZ G$-submodule of a
    nilpotent $\IZ G$-module of filtration length $r$ is again nilpotent and has
    filtration length $\le r$;

  \item \label{lem_filtration_length:quotientmodule} A $\IZ G$-quotient module of
    a nilpotent $\IZ G$-module of filtration length $r$ is again nilpotent and
    has filtration length $\le r$;

  \item \label{lem_filtration_length:exact_sequence} 
    Let $0 \to M_0 \to M_1 \to M_2 \to 0$ be an exact sequence of $\IZ G$-modules. 
    If two of them are nilpotent,
    then all three are nilpotent and the filtration length of $M_1$ is less or
    equal to the sum of the filtration length of $M_0$ and $M_2$;

  \item \label{lem_filtration_length:direct_sums} If $\{M_i \mid i \in I\}$
   is a family of nilpotent $\IZ G$-modules of filtration length $\le r$
   for some index set $I$, then $\bigoplus_{i \in I} M_i$ is a nilpotent 
   $\IZ G$-module of filtration length $\le r$.
  \end{enumerate}
\end{lemma}


 \typeout{--------------   Section 3: Proof of the main Theorem ----------------}

\section{Proof of the main Theorem~\ref{the:fibrations}}
\label{sec:Proof_of_the_main_Theorem}

In this section we present the proof of Theorem~\ref{the:fibrations}.

The following Subsections~\ref{subsec:Passage_from_H_n(C)_otimes_ZG_Z_to_H_n(C_otimes_ZG_Z)},~%
\ref{subsec:Passage_from_M_to_M_otimes_ZG_Z} and~\ref{subsec:Relating_H_n(C)_to_H_n(C_otimes_ZG_Z} deal with the problem
to give estimates of our invariants of interest for a $\IZ G$-chain complex $C_*$  with nilpotent homology in terms of their values for
$\IZ \otimes_{\IZ G} C_*$, where $G$ is a finite abelian group. The main result will be 
Proposition~\ref{pro:estimate_for_mg}.

\subsection{Passage from $ \IZ \otimes_{\IZ G}H_n(C_*)$ to $H_n(\IZ  \otimes_{\IZ G} C_*)$}
\label{subsec:Passage_from_H_n(C)_otimes_ZG_Z_to_H_n(C_otimes_ZG_Z)}

\begin{lemma} \label{lem:nu}
Let $G$ be a finite group. Let $C_*$ be a finitely generated free $\IZ G$-chain complex. 
Let $\nu_n = \nu_n(C_*)  \colon \IZ  \otimes_{\IZ G}  H_n(C_*) \to H_n(\IZ  \otimes_{\IZ G}  C_*)$ be the canonical map.

Then
\begin{eqnarray*}
  \bigl|\ker(\nu_n)\bigr| 
  & \le &
  \prod_{p=1}^n \bigl|H_{p+1}(G;H_{n-p}(C_*))\bigr|;
  \\
  \bigl|\coker(\nu_n)\bigr| 
  & \le &
  \prod_{p=1}^n \bigl|H_p(G;H_{n-p}(C_*))\bigr|;
  \\
  \left|\mg\bigl(\IZ  \otimes_{\IZ G} H_n(C_*)\bigr) 
  - \mg\bigl(H_n(\IZ  \otimes_{\IZ G} C_*)\bigr)\right|
  & \le &
  \sum_{p=1}^n \mg\bigl(H_p(G;H_{n-p}(C_*))\bigr) 
  \\
  & & \hspace{1mm} + \sum_{p=1}^n  \mg\bigl(H_{p+1}(G;H_{n-p}(C_*))\bigr).
\end{eqnarray*}
\end{lemma}
\begin{proof}
  The map $\nu_n \colon \IZ  \otimes_{\IZ G} H_n(C_*)  = \Tor^{\IZ
    G}_0(H_n(C_*),\IZ) \to H_n(\IZ  \otimes_{\IZ G}  C_*)$ is the edge
  homomorphism in the universal coefficient spectral sequence whose $E^2$-term
  is $\Tor^{\IZ G}_p(H_q(C_*),\IZ) = H_p(G;H_q(C_*))$ and which converges to
  $H_n(\IZ  \otimes_{\IZ G}  C_*)$. From this spectral sequence we obtain an exact
  sequence of abelian groups
  \[
  0 \to A \to \IZ  \otimes_{\IZ G}  H_n(C_*)  \xrightarrow{\nu_n} H_n(\IZ  \otimes_{\IZ G}  C_*) \to B \to 0,
  \]
  where the modules $A$ and $B$ have filtrations $\{0\} \subseteq A_1 \subseteq
  \cdots \subseteq A_n = A$ and $\{0\} \subseteq B_1 \subseteq \cdots\subseteq  B_n = B$
  such that $A_n$ is a sub-quotient of $H_{p+1}(G;H_{n-p}(C_*))$ and
  $B_n$ is a sub-quotient of $H_p(G;H_{n-p}(C_*))$ for each $p \in
  \{1,2, \ldots, n\}$. In particular we get from
  Lemma~\ref{lem:properties_of_mg}~\eqref{lem:properties_of_mg:subadditivity}
  \begin{eqnarray*}
    |A| 
    & \le & 
    \prod_{p=1}^n \bigl|H_{p+1}(G;H_{n-p}(C_*))\bigr|;
    \\
    |B| 
    & \le & 
    \prod_{p=1}^n \bigl|H_{p}(G;H_{n-p}(C_*))\bigr|;
    \\
    \mg(A)
    & \le & 
    \sum_{p=1}^n \mg\bigl(H_{p+1}(G;H_{n-p}(C_*))\bigr);
    \\
    \mg(B) 
    & \le & 
    \sum_{p=1}^n \mg\bigl(H_{p}(G;H_{n-p}(C_*))\bigr);
    \\
    \mg(A) 
    & \ge & 
    \mg\bigl( \IZ  \otimes_{\IZ G} H_n(C_*)\bigr) 
    - \mg\bigl(H_n(\IZ  \otimes_{\IZ G} C_*)\bigr);
    \\
    \mg(B) 
    & \ge & 
    \mg\bigl( H_n(\IZ  \otimes_{\IZ G} C_*)\bigr) 
    - \mg\bigl( \IZ  \otimes_{\IZ G}  H_n(C_*)\bigr). 
  \end{eqnarray*}
  Now Lemma~\ref{lem:nu} follows.
\end{proof}

\subsection{Passage from $M$ to $\IZ  \otimes_{\IZ G} M $}
\label{subsec:Passage_from_M_to_M_otimes_ZG_Z}

\begin{lemma} \label{lem:mu} Let $G$ be a finite group.
  Let $M$ be a nilpotent $\IZ G$-module. Let $r \ge
  1$ be an integer such that the filtration length of $M$ is less or equal to
  $r$. Then the kernel of the canonical epimorphism 
  $\mu = \mu(M) \colon M \to \IZ  \otimes_{\IZ G} M$ is finite and we have
  \begin{eqnarray*}
    \bigl|\ker(\mu(M))\bigr| 
    & \le &
    |G|^{(r-1) \cdot \mg(G) \cdot \mg(M)};
    \\
    \mg(M) 
    & \le &
    r \cdot (\mg(G) + 1)^{r-1} \cdot \mg(\IZ  \otimes_{\IZ G} M).
  \end{eqnarray*}
\end{lemma}
\begin{proof}
  We use induction over the filtration length $r$ of $M$. The induction
  beginning $r = 1$ is obviously true.  The induction step from $r-1$ to 
  $r\ge 2$ is done as follows.

  Choose an exact sequence of $\IZ G$-modules $0 \to M' \to M \to M'' \to 0$
  such that $M'$ is nilpotent of filtration length $\le r-1$ and $G$ acts
  trivially on $M''$.  We have the following commutative diagram with exact rows
\[
\xymatrix{ 0 \ar[r] 
& 
M' \ar[r] \ar[d]^{\mu(M')} 
& 
M \ar[r] \ar[d]^{\mu(M)} 
& 
M'' \ar[r] \ar[d]^{\mu(M'')}_{\cong}
& 0
\\
H_1(G;M'') \ar[r] 
& 
\IZ  \otimes_{\IZ G}  M' \ar[r]
& 
\IZ  \otimes_{\IZ G} M  \ar[r]
& 
\IZ  \otimes_{\IZ G}  M'' \ar[r]
&
0
}
\]
It induces the exact sequence
\begin{eqnarray}
& & 0 \to  \ker(\mu(M')) \to \ker(\mu(M)) \to \im\bigl(H_1(G;M'') 
\to \IZ  \otimes_{\IZ G} M'\bigr) \to 0.
\label{lem:mu:exact_sequence}
\end{eqnarray}
Let $\{s_1, s_2, \ldots, s_{\mg(G)}\}$ be a finite set of generators of $G$ with
minimal numbers of generators.  We obtain an exact sequence of $\IZ G$-modules
\[
\IZ G^{\mg(G)} \xrightarrow{(s_1 -1, s_2 -1, \ldots , s_{\mg(G)} -1)} 
\IZ G \to \IZ \to 0,
\]
which can be extended to the left to a free $\IZ G$-resolution of the trivial
$\IZ G$-module $\IZ$.  Because of
Lemma~\ref{lem:properties_of_mg}~\eqref{lem:properties_of_mg:subadditivity} this
implies
\begin{multline}
  \mg\bigl(H_1(G,M'')\bigr) \le \mg\bigl(\IZ G^{\mg(G)} \otimes_{\IZ G}
  M''\bigr) \le \mg(G) \cdot \mg(M'')
  \\
  = \mg(G) \cdot \mg(\IZ  \otimes_{\IZ G} M'') \le \mg(G) 
  \cdot \mg(\IZ  \otimes_{\IZ G} M).
  \label{lem:mu:(1)}
\end{multline}

As multiplication with $G$ annihilates $H_1(G,M)$ we get
\begin{eqnarray}
|H_1(G,M'')| 
& \le & 
|G|^{\mg(\IZ  \otimes_{\IZ G} M) \cdot \mg(G)}.
\label{lem:mu:(2)}
\end{eqnarray}
This implies
\begin{eqnarray}
|H_1(G,M'')| 
& \le & 
|G|^{\mg(M) \cdot \mg(G)}.
\label{lem:mu:(3)}
\end{eqnarray}
By induction hypothesis
\begin{eqnarray}
\bigl|\ker(\mu(M'))\bigr| 
& \le & 
|G|^{(r-2) \cdot \mg(G) \cdot \mg(M')}.
\label{lem:mu:(4)}
\end{eqnarray}
We conclude from~\eqref{lem:mu:exact_sequence},~\eqref{lem:mu:(3)} 
and~\eqref{lem:mu:(4)} using Lemma~\ref{lem:properties_of_mg}~\eqref{lem:properties_of_mg:subadditivity}
\begin{eqnarray*}
\bigl|\ker(\mu(M))\bigr| 
& \le &
\bigl|\ker(\mu(M'))\bigr| \cdot |H_1(G;M'')| 
\\
& \le &
|G|^{(r-2) \cdot \mg(G) \cdot \mg(M')} \cdot |G|^{\mg(M'') \cdot \mg(G)}
\\
& \le &
|G|^{(r-2) \cdot \mg(G) \cdot \mg(M)} \cdot |G|^{\mg(M) \cdot \mg(G)}
\\
& = & 
|G|^{(r-1) \cdot \mg(G) \cdot \mg(M)}.
\end{eqnarray*}
Next  we estimate using 
Lemma~\ref{lem:properties_of_mg}~\eqref{lem:properties_of_mg:subadditivity},%
~\eqref{lem:mu:exact_sequence} and~\eqref{lem:mu:(1)} 
\begin{eqnarray}
\label{lem:mu:(5)} \quad 
\mg(M) 
& \le &
\mg\bigl(\ker(\mu(M))\bigr) + \mg(\IZ  \otimes_{\IZ G} M)
\\
& \le &
\mg\bigl(\ker(\mu(M'))\bigr) + \mg(H_1(G;M'')) + \mg(\IZ  \otimes_{\IZ G} M)
\nonumber
\\
& \le &
\mg\bigl(\ker(\mu(M'))\bigr) + \mg(G) 
\cdot \mg(\IZ  \otimes_{\IZ G} M) + \mg(\IZ  \otimes_{\IZ G} M)
\nonumber
\\
& \le &
\mg(M') + (\mg(G) + 1) \cdot \mg(\IZ  \otimes_{\IZ G} M). 
\nonumber
\end{eqnarray}
We estimate using~\eqref{lem:mu:(1)} and Lemma~\ref{lem:properties_of_mg}~\eqref{lem:properties_of_mg:subadditivity}
\begin{eqnarray}
\label{lem:mu:(6)} 
\mg(\IZ  \otimes_{\IZ G} M') 
& \le & 
\mg(H_1(G;M''))  + \mg(\IZ  \otimes_{\IZ G} M) 
\\
& \le &
 \mg(G) \cdot \mg(\IZ  \otimes_{\IZ G} M) + \mg(\IZ  \otimes_{\IZ G} M) 
\nonumber
\\
& = & 
(\mg(G) +1 ) \cdot \mg(\IZ  \otimes_{\IZ G} M).
\nonumber
\end{eqnarray}
By induction hypothesis
\begin{eqnarray}
\mg(M') 
& \le &
(r-1) \cdot  (m(G)+1)^{r-2} \cdot \mg(\IZ  \otimes_{\IZ G} M').
\label{lem:mu:(7)}
\end{eqnarray}
We conclude from~\eqref{lem:mu:(5)},~\eqref{lem:mu:(6)}  and~\eqref{lem:mu:(7)}
\begin{eqnarray*}
\mg(M) 
& \le &
\mg(M') + (\mg(G) + 1) \cdot \mg(\IZ  \otimes_{\IZ G} M)
\\
& \le &
(r-1) \cdot  (m(G)+1)^{r-2} \cdot \mg(\IZ  \otimes_{\IZ G} M') + (\mg(G) + 1) 
\cdot \mg(\IZ  \otimes_{\IZ G} M)
\\
& \le &
(r-1) \cdot  (m(G)+1)^{r-2} \cdot (\mg(G) +1 ) \cdot \mg(\IZ  \otimes_{\IZ G} M) 
\\
& & 
\hspace{30mm} + (\mg(G) + 1) \cdot \mg(\IZ  \otimes_{\IZ G} M)
\\
& = &
\left((r-1) \cdot  (m(G)+1)^{r-2} +1\right) \cdot (\mg(G) +1 ) 
\cdot \mg(\IZ  \otimes_{\IZ G} M) 
\\
& \le &
r \cdot  (m(G)+1)^{r-2}\cdot (\mg(G) +1 ) \cdot \mg(\IZ  \otimes_{\IZ G} M) 
\\
& = &
r \cdot  (m(G)+1)^{r-1} \cdot \mg(\IZ  \otimes_{\IZ G} M).
\end{eqnarray*}
This finishes the proof of Lemma~\ref{lem:mu}.
\end{proof}

\begin{lemma}\label{lem:finite_quotients_of_Zm}
Let $G$ be a finite abelian group.
Then there exists a free $\IZ G$-resolution $F_*$  of the trivial
$\IZ G$-module $\IZ$ satisfying for $n \ge 0$
\[\dim_{\IZ G}(F_n) =  \binom{n + \mg(G)-1}{\mg(G)-1}.
\]
\end{lemma}
\begin{proof}
  Put $m = \mg(G)$. Then we get decomposition
  \[G = \bigoplus_{i = 1}^m \IZ/d_i
  \]
  as direct sum of finite cyclic groups from
  Lemma~\ref{lem:properties_of_mg}~\eqref{lem:properties_of_mg:structure_d}.
   There is a free $\IZ[\IZ/d_i]$-resolution
  $F^{(i)}_*$ of the trivial $\IZ[\IZ/d_i]$-module $\IZ$ such that
  $\dim_{\IZ[\IZ/d_i]}(F^{(i)}_n) = 1$ for all $n \ge 0$. Then 
  $F_* = \bigotimes_{i = 1}^m F^{(i)}_*$ is a free $\IZ G$-resolution of the trivial
  $\IZ G$-module $\IZ$. Obviously $\dim_{\IZ G}(F_n)$ is the number of weak compositions 
  of $n$ into $m$ non-negative integers, i.e., the number of $k$-tuples $(i_1,i_2, \ldots, i_m)$
  of non-negative integers such that $i _1 + i_2 + \cdots + i_m = n$.  Hence
  \[
  \dim_{\IZ G}(F_n) = \binom{n + m-1}{m-1}. 
  \]
\end{proof}

\begin{remark}
Lemma~\ref{lem:finite_quotients_of_Zm} is the reason why we often assume the existence of an infinite
normal subgroup $A \subseteq G$ which is abelian. The estimate we get and will need is polynomial in $\mg(G)$.
There is an obvious polynomial estimate in terms of $|G|$ coming from the bar resolution
which works for all finite groups, but it is not sufficient for us.
\end{remark}

\subsection{Relating $H_n(C_*)$ to $H_n(\IZ  \otimes_{\IZ G} C_*)$}
\label{subsec:Relating_H_n(C)_to_H_n(C_otimes_ZG_Z}

\begin{proposition} \label{pro:estimate_for_mg} There exists functions
  \begin{eqnarray*}
    C_0,C_1,D_0,D_1 \colon \bigl\{(r,n,p) \in \IN^3\mid p \le n\bigr\} & \to & \IR, \quad
  \end{eqnarray*}
  such that the following is true:

  Let $G$ be a finite abelian group. Let $C_*$ be a finitely generated free
  $\IZ G$-chain complex.  Suppose that $H_p(C_*)$ is a nilpotent $\IZ G$-module
  of filtration length $ \le r$ for $p = 0,1,2, \ldots , n$. Then:
  \begin{enumerate}
  \item \label{pro:estimate_for_mg:mg} We have
    \[
    \mg(H_n(C_*)) \le \sum_{p = 0}^n C_0(r,n,p) \cdot \mg(G)^{C_1(r,n,p)} \mg\bigl(H_p(\IZ  \otimes_{\IZ G} C_*)\bigr);
    \]

  \item \label{pro:estimate_for_mg:order} We have
    \begin{eqnarray*}
      \ln\left(\bigl|\ker(H_n(\pr_*))\bigr|\right)
      & \le &
      \sum_{p = 0}^n D_0(r,n,p) \cdot \ln(|G|) \cdot  \mg(G)^{D_1(r,n,p)} 
      \cdot \mg\bigl(H_p(\IZ  \otimes_{\IZ G} C_*)\bigr);
      \\
      \ln\left(\bigl|\coker(H_n(\pr_*))\bigr|\right)
      & \le &
      \sum_{p = 0}^n D_0(r,n,p) \cdot \ln(|G|) \cdot  \mg(G)^{D_1(r,n,p)}  
      \cdot \mg\bigl(H_p(\IZ  \otimes_{\IZ G} C_*)\bigr),
    \end{eqnarray*}
    where $\pr_* \colon C_*\to \IZ  \otimes_{\IZ G} C_*$ is the canonical
    projection.
  \end{enumerate}
\end{proposition}
\begin{proof}~\eqref{pro:estimate_for_mg:mg} Fix a natural number $r$. We define $C(r,n,p)$ for $p \in \{0,1,2,\ldots, n\}$
  by induction over $n$. In the sequel we abbreviate $m = \mg(G)$.

  Firstly we explain the induction beginning $n = 0$. Then 
  $\nu_0 \colon  \IZ  \otimes_{\IZ G}  H_0(C_*) \to H_0(\IZ  \otimes_{\IZ G} C_*)$ is an
  isomorphism since we always assume $C_n = 0$ for $p \le -1$.  We conclude from
  Lemma~\ref{lem:mu}
  \begin{eqnarray*}
    \mg(H_0(C_*)) 
    & \le & 
    r \cdot (m+1)^{r-1} \cdot \mg\bigl(\IZ  \otimes_{\IZ G} H_0(C_*)\bigr)
    \\
    & = & 
    r \cdot (m +1)^{r-1} \cdot \mg\bigl(H_0(\IZ  \otimes_{\IZ G} C_*)\bigr)
    \\
     & \le & 
    r \cdot 2^{r-1} \cdot  m^{r-1} \cdot \mg\bigl(H_0(\IZ  \otimes_{\IZ G} C_*)\bigr).
  \end{eqnarray*}
  So we put $C_0(r,0,0) = r \cdot 2^{r-1} $ and $C_1(r,0,0) = r-1$.

  The induction beginning from $(n-1)$ for $n \ge 1$ is done as follows.  We
  conclude from Lemma~\ref{lem:mu}
  \begin{eqnarray}
    \mg(H_n(C_*)) 
    & \le &
    r \cdot (m +1)^{r-1} \cdot \mg\bigl(\IZ  \otimes_{\IZ G}  H_n(C_*)\bigr)
    \label{pro:estimate_for_mg:(1)}
    \\
    & \le &
    r \cdot 2^{r-1} \cdot m^{r-1} \cdot \mg\bigl(\IZ  \otimes_{\IZ G}  H_n(C_*)\bigr).
    \nonumber
  \end{eqnarray}
  We estimate using Lemma~\ref{lemma:order_of_group_homology},
  Lemma~\ref{lem:nu}, and Lemma~\ref{lem:finite_quotients_of_Zm}
  \begin{eqnarray}
    \label{pro:estimate_for_mg:(2)}
    \lefteqn{\mg\bigl(\IZ  \otimes_{\IZ G} H_n(C_*)\bigr)} 
    & & 
    \\
    & \le &
    \mg\bigl(H_n(\IZ  \otimes_{\IZ G} C_*) \bigr) + 
    \sum_{p=1}^n \mg\bigl(H_p(G;H_{n-p}(C_*))\bigr) 
    \nonumber
    \\
    & & \hspace{5mm} + \sum_{p=1}^n  \mg\bigl(H_{p+1}(G;H_{n-p}(C_*))\bigr)
    \nonumber
    \\
    & \le &
    \mg\bigl(H_n(\IZ  \otimes_{\IZ G} C_*) \bigr) + 
    \sum_{p=1}^n  \binom{p +m -1}{m-1}\cdot \mg\bigl(H_{n-p}(C_*))\bigr) 
    \nonumber
    \\
    & & \hspace{5mm} + \sum_{p=1}^n  \binom{p  + m}{m-1}\cdot \mg\bigl(H_{n-p}(C_*))\bigr)
    \nonumber
    \\
    & = &
    \mg\bigl(H_n(\IZ  \otimes_{\IZ G} C_*) \bigr)
    \nonumber
    \\
    & & \hspace{5mm} 
    + \sum_{p=0}^{n-1}  \left(\binom{n - p + m -1}{m-1} + \binom{n - p  + m}{m-1}\right)\cdot \mg\bigl(H_{p}(C_*))\bigr).
    \nonumber
  \end{eqnarray}
  Putting~\eqref{pro:estimate_for_mg:(1)} and~\eqref{pro:estimate_for_mg:(2)}
  together yields
  \begin{multline}
    \label{pro:estimate_for_mg:(3a)} 
    \mg\bigl(H_n(C_*)\bigr) 
    \le  
    r \cdot 2^{r-1} \cdot m^{r-1} \cdot \mg\bigl(H_n(\IZ  \otimes_{\IZ G}  C_*) \bigr)
    \\
    + \sum_{p=0}^{n-1}  r \cdot  2^{r-1} \cdot m ^{r-1} \cdot  
    \left(\binom{n - p + m -1 }{m-1} + \binom{n - p + m}{m-1}\right) \cdot \mg\bigl(H_{p}(C_*))\bigr).
 \end{multline}
  We estimate for $p = 0,1,2\ldots ,n$
  \begin{eqnarray}
   \label{binom_estimate_1}
   \binom{n - p + m }{m-1} 
   & = & 
    \frac{(n-p+m)!}{(m-1)! \cdot (n-p+1)!}
   \\
   & \le &
   (n-p+m) \cdot (n-p+m-1) \cdot \;\cdots \; \cdot m
   \nonumber
   \\
   & \le & (n-p+m)^{n-p+1}
   \nonumber
   \\
   & \le &
   (n+m)^{n+1}
   \nonumber
   \\
   & \le  &
   n^{n+1} \cdot m^{n+1}.
  \nonumber
  \end{eqnarray}
  Analogously we get for $p = 0,1,2\ldots ,n$
  \begin{eqnarray}
  \label{binom_estimate_2}
   \binom{n - p + m -1 }{m-1} 
   & \le &
   n^{n+1} \cdot m^{n+1}.
 \end{eqnarray}
  We conclude from~\eqref{pro:estimate_for_mg:(3a)},~\eqref{binom_estimate_1} 
  and~\eqref{binom_estimate_2}
    \begin{multline}
    \label{pro:estimate_for_mg:(3)} 
    \mg\bigl(H_n(C_*)\bigr) 
    \le  
    r \cdot 2^{r-1} \cdot m^{r-1} \cdot \mg\bigl(H_n(\IZ  \otimes_{\IZ G}  C_*) \bigr)
    \\
    + \sum_{p=0}^{n-1}  r \cdot 2^{r-1} \cdot m^{r-1}  \cdot 2 \cdot  n^{n+1} \cdot m^{n+1}  \cdot \mg\bigl(H_{p}(C_*))\bigr).
 \end{multline}
 By induction hypothesis we get for $p = 0,1,2, \ldots, (n-1)$
  \begin{eqnarray}
    \mg(H_p(C_*)) 
    &\le &
    \sum_{i = 0}^p C_0(r,p,i) \cdot m^{C_1(r,p,i)} \cdot \mg\bigl(H_i(\IZ  \otimes_{\IZ G} C_*)\bigr).
    \label{pro:estimate_for_mg:(4)}
  \end{eqnarray}
  Putting~\eqref{pro:estimate_for_mg:(3)} and~\eqref{pro:estimate_for_mg:(4)}
  together yields 
  \begin{eqnarray*}
    \mg\bigl(H_n(C_*)\bigr) 
    & \le & 
    r \cdot  2^{r-1} \cdot m^{r-1} \cdot \mg\bigl(H_n(\IZ  \otimes_{\IZ G} C_*) \bigr)
    \\
    & & \hspace{5mm} 
    + \sum_{p=0}^{n-1}  r \cdot  2^{r-1} \cdot m^{r-1} \cdot  2\cdot  n^{n+1} \cdot m^{n+1} \cdot 
    \\ & & \hspace{10mm} 
    \left(\sum_{i = 0}^p C_0(r,p,i) \cdot m^{C_1(r,p,i)} \cdot \mg\bigl(H_i(\IZ  \otimes_{\IZ G} C_*)\bigr)\right)
    \\
    & \le & 
    r \cdot  2^{r-1} \cdot m^{r-1} \cdot \mg\bigl(H_n(\IZ  \otimes_{\IZ G} C_*) \bigr)
    \\
    & & \hspace{5mm} 
    + \sum_{p = 0}^{n-1} \left(\sum_{i = p}^{n-1}  r \cdot  2^{r}  \cdot  n^{n+1} 
    \cdot  C_0(r,i,p) \cdot m^{n+r+C_1(r,i,p)}  \right) 
    \\ & & \hspace{10mm} 
    \cdot
    \mg\bigl(H_p(\IZ  \otimes_{\IZ G}  C_*)\bigr)
   \\
    & \le & 
    r \cdot  2^{r-1} \cdot m^{r-1} \cdot \mg\bigl(H_n(\IZ  \otimes_{\IZ G} C_*) \bigr)
    \\
    & & \hspace{5mm} 
    + \sum_{p = 0}^{n-1} \left(\sum_{i = p}^{n-1}  r \cdot  2^{r}  \cdot  n^{n+1} \cdot C_0(r,i,p) \right) 
     \\ & & 
    \hspace{10mm} 
    \cdot m^{n+r+\max\{C_1(r,i,p)\mid p \le i \le n-1)\}}  \cdot
    \mg\bigl(H_p(\IZ  \otimes_{\IZ G}  C_*)\bigr).
  \end{eqnarray*}
  Define
  \begin{eqnarray*}
   C_0(r,n,p)  
   & := & \
   \begin{cases}
    \sum_{i = p}^{n-1}  r \cdot  2^{r}  \cdot  n^{n+1}
    \cdot  C_0(r,i,p); & \text{if}\; 0 \le p \le n-1;
    \\
    r \cdot  2^{r-1}  & \text{if}\;  p = n;
  \end{cases}
  \\
   C_1(r,n,p) 
   & := & 
   \begin{cases}
   n+r+\max\{C_1(r,i,p)\mid p \le i \le n-1)\}; & \text{if}\;  0 \le p \le n-1;
   \\    
   r-1 & \text{if}\;  p = n.
 \end{cases}
\end{eqnarray*}
  Then  we get
  \[
  \mg(H_n(C_*)) \le \sum_{p = 0}^n C_0(r,n,p) \cdot m^{C_1(r,n,p)} \cdot\mg\bigl(H_n(\IZ  \otimes_{\IZ G} C_*)\bigr).
  \]
  This finishes the proof of assertion~\eqref{pro:estimate_for_mg:mg}.
  \\[2mm]~\eqref{pro:estimate_for_mg:order} 
  We estimate using
  Lemma~\ref{lemma:order_of_group_homology}, Lemma~\ref{lem:nu},  Lemma~\ref{lem:finite_quotients_of_Zm},
  and~\eqref{binom_estimate_2}
  \begin{eqnarray*}
    |\coker(\nu_n(C_*))| 
    & \le &
    \prod_{p=1}^n \bigl|H_p(G;H_{n-p}(C_*))\bigr|
    \\
    & \le &
    \prod_{p=1}^n |G|^{\binom{p+m-1}{m-1} \cdot \mg(H_{n-p}(C_*))}
    \\
    & = & 
    |G|^{\sum_{p=1}^n (\binom{p+m-1}{m-1}\cdot \mg(H_{n-p}(C_*))}
     \\
    & \le  & 
    |G|^{\sum_{p=1}^n n^{n+1} \cdot m^{n+1}\cdot \mg(H_{n-p}(C_*))}.
  \end{eqnarray*}

  This implies together with assertion~\eqref{pro:estimate_for_mg:mg}
  \begin{eqnarray}
    \label{pro:estimate_for_mg:(6)}
    & & 
    \\
    \lefteqn{\ln\bigl(|\coker(\nu_n(C_*))|\bigr)} & & 
    \nonumber
    \\
    & \le &
    \ln(|G|) \cdot \sum_{p=1}^n  n^{n+1} \cdot m^{n+1}\cdot \mg(H_{n-p}(C_*))
    \nonumber
    \\
    & \le &
    \ln(|G|) \cdot \sum_{p=1}^n n^{n+1} \cdot m^{n+1}  \cdot 
    \left(\sum_{i = 0}^{n-p} C_0(r,n-p,i) \cdot m^{C_1(r,n-p,i)}
    \mg\bigl(H_i(\IZ  \otimes_{\IZ G} C_*)\bigr)\right)
    \nonumber
    \\
    & = & 
    \ln(|G|) \cdot   \sum_{p= 0}^{n-1} \left(\sum_{i = 1}^{n-p} n^{n+1} \cdot m^{n+1} \cdot C_0(r,n-i,p) \cdot m^{C_1(r,n-i,p)} \right) \cdot
      \mg\bigl(H_p(\IZ  \otimes_{\IZ G} C_*)\bigr)
    \nonumber
    \\
    & = & 
    \sum_{p= 0}^{n-1} \left(\sum_{i = 1}^{n-p} n^{n+1} \cdot  C_0(r,n-i,p) \cdot \ln(|G|) \cdot   m^{n+1 + C_1(r,n-i,p)} \right) \cdot 
      \mg\bigl(H_p(\IZ  \otimes_{\IZ G} C_*)\bigr).
    \nonumber 
  \end{eqnarray}
  Analogously we get
  \begin{eqnarray}
    \label{pro:estimate_for_mg:(7)} & & 
    \\
    \lefteqn{\ln\bigl(|\ker(\nu_n(C_*))|\bigr)}
    \nonumber
    \\
    & \le & 
    \sum_{p= 0}^{n-1} \left(\sum_{i = 1}^{n-p} n^{n+1} \cdot  C_0(r,n-i,p) \cdot \ln(|G|) \cdot  m^{n+1+C_1(r,n-i,p)} \right) \cdot 
      \mg\bigl(H_p(\IZ  \otimes_{\IZ G} C_*)\bigr).
    \nonumber
  \end{eqnarray}
  We estimate using Lemma~\ref{lem:mu} and  assertion~\eqref{pro:estimate_for_mg:mg}.
  \begin{eqnarray}
     \label{pro:estimate_for_mg:(8)} 
    & & 
    \\
    \lefteqn{\ln\bigl(\bigl|\ker(\mu(H_n(C_*))\bigr|\bigr)}
    \nonumber
    \\ 
     & \le &
    \ln(|G|) \cdot (r-1) \cdot m \cdot \mg(H_n(C_*))
    \nonumber
    \\
    & \le &
    \ln(|G|) \cdot (r-1) \cdot m \cdot \sum_{p = 0}^n C_0(r,n,p) \cdot m^{C_1(m,n,p)} \cdot
    \mg\bigl(H_p(\IZ  \otimes_{\IZ G} C_*)\bigr)
    \nonumber
    \\
    & \le &
    \sum_{p = 0}^n (r-1) \cdot C_0(r,n,p) \cdot \ln(|G|) \cdot m^{1 + C_1(m,n,p)} \cdot
    \mg\bigl(H_p(\IZ  \otimes_{\IZ G} C_*)\bigr).
    \nonumber
    \nonumber
  \end{eqnarray}
  Since the canonical map $H_n(\pr_*) \colon H_n(C_*) \to H_n(\IZ  \otimes_{\IZ G} C_*)$ 
agrees with the composition $\nu_n(C_*) \circ \mu(H_n(C_*))$ and
  $\mu_n(H_n(C_*))$ is surjective, we obtain an exact sequence
  \begin{eqnarray*}
    & 0 \to \ker\bigl(\mu(H_n(C_*))\bigr) \to \ker(H_n(\pr_*)) 
    \to \ker\bigl(\nu(C_*))\bigr)  \to 0, &
  \end{eqnarray*}
  and an isomorphism
  \begin{eqnarray*}
    \coker(H_n(\pr_*)) & \cong & \coker\bigl(\nu(C_*))\bigr).
  \end{eqnarray*}
  We conclude
  \begin{eqnarray*}
    \ln\left(\bigl|\ker(H_n(\pr_*))\bigr|\right)
    & \le &
    \ln\left(\bigl|\ker\bigl(\mu(H_n(C_*))\bigr)\bigr|\right)  
    + \ln\left(\bigl|\ker\bigl(\nu_n(C_*))\bigr)\bigr|\right);
    \\
    \ln\left(\bigl|\coker(H_n(\pr_*))\bigr|\right) 
    & = & 
    \ln\left(\bigl|\coker\bigl(\nu_n(C_*))\bigr)\bigr|\right).
  \end{eqnarray*}
  This together with~\eqref{pro:estimate_for_mg:(6)},
  \eqref{pro:estimate_for_mg:(7)}, and~\eqref{pro:estimate_for_mg:(8)} implies
  \begin{eqnarray*}
    \lefteqn{\ln\left(\bigl|\ker(H_n(\pr_*))\bigr|\right)} 
    & & 
    \\
    &\le &
    \sum_{p = 0}^n (r-1) \cdot C_0(r,n,p) \cdot \ln(|G|) \cdot m^{1 + C_1(m,n,p)} \cdot
    \mg\bigl(H_p(\IZ  \otimes_{\IZ G} C_*)\bigr)
    \\
    & & 
    \hspace{7mm} +   
    \sum_{p= 0}^{n-1} \left(\sum_{i = 1}^{n-p} n^{n+1} \cdot  C_0(r,n-i,p) \cdot m^{n+1+C_1(r,n-i,p)} \right) \cdot \ln(|G|)
    \\
    & & 
    \hspace{14mm}\cdot \mg\bigl(H_p(\IZ  \otimes_{\IZ G} C_*)
    \\
    & \le &
    \sum_{p = 0}^n (r-1) \cdot C_0(r,n,p) \cdot \ln(|G|) \cdot m^{1 + C_1(m,n,p)} \cdot
    \mg\bigl(H_p(\IZ  \otimes_{\IZ G} C_*)\bigr)
    \\
    & & \hspace{7mm} +   
    \sum_{p= 0}^{n-1} \left(\sum_{i = 1}^{n-p}  n^{n+1} \cdot  C_0(r,n-i,p) \right) \cdot  \ln(|G|)  
    \\
   & & \hspace{14mm} \cdot m^{n+1+\max\{C_1(r,n-i,p)\mid i = 1, \ldots n-p\}}  \cdot \mg\bigl(H_p(\IZ  \otimes_{\IZ G} C_*)\bigr),
 \end{eqnarray*}
 and
  \begin{eqnarray*}
    \lefteqn{\ln\bigl(|\coker(H_n(\pr_*))|\bigr)}
    & &
     \\
    & \le & 
     \sum_{p= 0}^{n-1} \left(\sum_{i = 1}^{n-p} n^{n+1} \cdot  C_0(r,n-i,p) \right)  \cdot   \ln(|G|)
     \\
     & &
     \hspace{7mm}  \cdot m^{n+1 + \max\{C_1(r,n-i,p)\mid i = 1, \ldots n-p\}}  \cdot  \mg\bigl(H_p(\IZ  \otimes_{\IZ G} C_*)\bigr).
   \end{eqnarray*}
   Now assertion~\eqref{pro:estimate_for_mg:order} follows for the obvious choices of $D_0(r,n,p)$ and $D_1(r,n,p)$. 
  This finishes the proof of Proposition~\ref{pro:estimate_for_mg}.
\end{proof}

\subsection{Nilpotent homology}
\label{subsec:nilpotent_homology}

In this subsection we explain how the condition that our $CW$-complex under consideration fibers in a specific
way enters the proof. Essentially the condition about Gottlieb's subgroup of the fundamental group ensures
that later we get on each level nilpotent homology groups what will be needed in order to apply Proposition~\ref{pro:estimate_for_mg}.

\begin{lemma} \label{lem:fibrations} 
  Let $F \xrightarrow{j} X \xrightarrow{f} B$ be a fibration, 
  where $F$ and $B$  are  connected $CW$-complexes. 
  Let $d$ be the dimension of $B$.  Consider an epimorphism
  $\phi \colon \pi_1(X) \to G$. Let $A \subseteq G$ be the
  image of $G_1(F)$ under the composite $\phi \circ \pi_1(j) \colon \pi_1(F) \to G$.  Let 
  $p \colon \overline{X} \to X$ be the $G$-covering associated to $\phi$.

  Then $A$ is a normal abelian subgroup of $G$ and the induced $A$-action on $\overline{X}$ turns 
  $H_n(\overline{X})$  into a $\IZ A$-module which is nilpotent of filtration 
  length $\le d+1$.
\end{lemma}
\begin{proof}
  We use induction over $d$. The induction beginning $d = -1$ is trivial since
  then $\overline{X}$ is empty.  The induction step from $(d-1)$ to $d \ge 0$ is
  done as follows.  Since $B$ is connected, it is homotopy equivalent to a
  finite $CW$-complex with precisely one $0$-cell.  Hence we can assume without
  loss of generality that each skeleton $B_d$ is connected.  Choose a pushout
  \[
  \xycomsquareminus{\coprod_{i \in I} S^{d-1}}{\coprod_{i \in I}
    q_i}{B_{d-1}}{}{}{\coprod_{i \in I} D^d}{\coprod_{i \in I} Q_i}{B}
  \]
  Since $f$ is a fibration and $p$ is a $G$-covering, we obtain by the pullback
  construction a $G$-pushout with a cofibration as left vertical arrow
  (see~\cite[Lemma 1.26]{Lueck(1989)})
  \[
  \xycomsquareminus{\coprod_{i \in I} q_i^*\overline{X}} {\coprod_{i \in I}
    \overline{q_i}}{(f \circ p)^{-1}(B_{d-1})} {}{} {\coprod_{i \in I}
    Q_i^*\overline{X}}{\coprod_{i \in I} \overline{Q_i}}{\overline{X}}
  \]
  Let $w \in \pi_1(X)$. Then the pointed fiber transport along
  $w$ yields a pointed homotopy equivalence $\sigma(w) \colon F \to F$ which is
  unique up to pointed homotopy equivalence
  (see~\cite[Section~6]{Lueck(1986)}). The induced map $\pi_1(\sigma(w))$  sends $G_1(F)$ to $G_1(F)$ by
  Lemma~\ref{lem:characterizing_G_1(F)}~%
\eqref{lem:characterizing_G_1(F):homotopy_invariance}.  The following diagram commutes
  \[
  \xycomsquareminus{\pi(F)}{\pi_1(\sigma(w))}{\pi_1(F)}{\pi_1(j)}{\pi_1(j)}
  {\pi_1(X)}{c_w}{\pi_1(X)}
  \]
  where $c_w$ is conjugation with $w$. Since $\phi \colon \pi_1(X) \to G$ is by
  assumption surjective, this implies that $A \subseteq G$ is normal.
  
  Let $H$ be the image of $\phi \circ \pi_1(j) \colon \pi_1(F) \to G$. Obviously
  $A \subseteq H$.  Let $\overline{F} \to F$ be the covering associated to the epimorphism
  $\pi_1(F) \to H$ induced by $\phi \circ \pi_1(j)$.
  Since $D^d$ is contractible, we obtain for each $i \in I$ a
  $G$-homotopy equivalence of pairs
  \[
  \bigl(Q_i^*\overline{X},q_i^*\overline{X}\bigr)
  \xrightarrow{\simeq} G \times _H\overline{F} \times (D^d,S^{d-1}).
  \]
  Hence we obtain a long
  exact sequence of $\IZ G$-modules
  \begin{multline*}
    \cdots \to \bigoplus_{i \in I} \IZ G \otimes_{\IZ H} H_{k}(\overline{F}) 
    \to H_{d+k}\bigl((f\circ p)^{-1}(B_{d-1})\bigr) \to H_{d+ k}(\overline{X})
    \\
    \to  \bigoplus_{i \in I} \IZ G \otimes_{\IZ H} H_k(\overline{F}) \to \cdots.
  \end{multline*}
  Now we view this as an exact sequence of $\IZ A$-modules. By induction
  hypothesis $H_{n}\bigl((f \circ p)^{-1}(B_{d-1})\bigr)$ is nilpotent of filtration
  length $\le d$ for all $n$.  The $A$-operation on $H_{l}(\overline{F})$ is
  trivial for all $l$, since for each $w \in \pi_1(F)$ the $\pi_1(F)$-map
  $l_w \colon \widetilde{F} \to \widetilde{F}$ given by multiplication with
  $w$ is $\pi_1(F)$-equivariantly homotopic to the identity because of
  Lemma~\ref{lem:characterizing_G_1(F)}~\eqref{lem:characterizing_G_1(F):universal_covering}.
  Hence the left $A$-action on $\IZ G \otimes_{\IZ H} H_{n}(\overline{F})$ 
  is trivial since $A$ is normal in $G$. This implies that 
  $\IZ G \otimes_{\IZ H} H_{n}(\overline{X})$ is a nilpotent $\IZ A$-module of filtration length 
  $\le 1$. We conclude from Lemma~\ref{lem_filtration_length}
  that $H_{n}(\overline{X})$ is a nilpotent $\IZ A$-module of
  filtration length $\le (d+1)$ for all $n$.
\end{proof}

\subsection{A priori bounds}
\label{subsec:A_nriori_bounds}

If the sequences under consideration converge,
they have to be bounded what we prove next.

\begin{notation}\label{not:C[i]} Let $C_*$ be a finite based free chain complex with differential $c_*$.
Denote by $(C[i]_*,c[i]_*)$ the finite based free $\IZ$-chain complex 
$\IZ \otimes_{\IZ G_i} C_* = \IZ[G/G_i] \otimes_{\IZ G} C_*$.
\end{notation}

\begin{lemma}\label{lem:boundedness_of_sequences}
    Consider a finite based free $\IZ G$-chain complex $C_*$.  Then
    there is a constant $\Lambda > 0$ satisfying:

  \begin{enumerate}

  \item \label{lem:boundedness_of_sequences:mg}
   For all $i \in I$ and $n \ge 0$ we have
   \[
    0 \le \frac{\mg(H_n(C[i]_*))}{[G:G_i]}
     \le \Lambda;
    \]

  \item \label{lem:boundedness_of_sequences:det(c[i]_n)}
   For all $i \in I$ and $n \ge 0$ we have
   \[
    0 \le \frac{\ln\bigl({\det}_{\caln(\{1\})}\bigl(c[i]_n^{(2)}\bigr)\bigr)}{[G:G_i]}
     \le \Lambda;
    \]

  \item   \label{lem:boundedness_of_sequences:IZ-torsion}
   For all $i \in I$ and $n \ge 0$ we have
   \[ \mg\bigl(H_n(C[i]_*)\bigr) \le \Lambda,
   \]
   and 
   \[
    0 \le \frac{\ln\left(\bigl|\tors(H_n(C[i]_*))\bigr|\right)}{[G:G_i]}
     \le \Lambda;
    \]

   \item \label{lem:boundedness_of_sequences:alpha}
    For all $i \in I$ and $n \ge 0$ we have
   \begin{eqnarray*}
   & - \Lambda \le \frac{\ln\bigl({\det}_{\caln(\{1\})}(\alpha[i]_n)\bigr)}{[G:G_i]} \le  \Lambda.
   \end{eqnarray*}
   where $\alpha[i]_n$ is the map associated to $C[i]_*$  in~\eqref{iso_alpha}.
 \end{enumerate}
\end{lemma}
\begin{proof}
We conclude from~\cite[Lemma~13.33 on page~466]{Lueck(2002)} that 
for every natural number $n$ there exists a constant 
$K_n \ge  0$ such that 
\begin{eqnarray}
||c[i]_n^{(2)}|| & \le & K_n
\label{order_of_c_by_K_n}
\end{eqnarray}
holds for each $i \in I$. Put 
\begin{eqnarray*}
\Lambda_0 
& := & 
\sum_{n = 0}^{\dim(C_*)} \max\{\ln(K_n),1\} \cdot \dim_{\IZ G}(C_n);
\\
\Lambda
& = &
4 \cdot \Lambda_0.
\end{eqnarray*}
Now we can prove the various assertions appearing in
Lemma~\ref{lem:boundedness_of_sequences}.
\\[1mm]~\eqref{lem:boundedness_of_sequences:mg}
We conclude from Lemma~\ref{lem:properties_of_mg}~\eqref{lem:properties_of_mg:subadditivity}
\begin{eqnarray*}
\frac{\mg(H_n(C_*))}{[G:G_i]} 
& \le & 
\frac{\dim_{\IZ}\bigl(\dim_{\IZ}(C[i]_n)\bigr)}{[G:G_i]}
\\
& = &
\dim_{\IZ G}(C_n)
\\
&\le &
\Lambda_0.
\end{eqnarray*}%
\eqref{lem:boundedness_of_sequences:det(c[i]_n)}
We conclude from from~\eqref{order_of_c_by_K_n}
\begin{eqnarray*}
{\det}_{\caln(\{1\})}\bigl(c[i]_n^{(2)}\bigr) 
& \le &
||c[i]_n^{(2)}||^{\dim_{\caln(\{1\}}(C[i]_n^{(2)})}
\nonumber\\
& \le &
(K_n)^{\dim_{\IZ}(C[i]_n)}.
\end{eqnarray*}
We have $1 \le {\det}_{\caln(\{1\}}\bigl(c[i]_n^{(2)}\bigr)$
since the trivial groups satisfy the Determinant Conjecture 
because of~\cite[Lemma~13.12 on page~459]{Lueck(2002)}.
Hence we get
\begin{eqnarray*}
& 1 \le {\det}_{\caln(\{1\})}\bigl(c[i]_n^{(2)}\bigr) 
\le 
(K_n)^{\dim_{\IZ}(C[i]_n)}.&
\end{eqnarray*}
Since 
$\dim_{\IZ G}(C_n) = \frac{\dim_{\IZ}(C[i]_n)}{[G:G_i]}$,
this implies for all $i \in I$ and $n \ge 0$
\begin{eqnarray}
& 0 \le \frac{\ln({\det}_{\caln(\{1\})}(c[i]_n^{(2)}))}{[G:G_i]} 
\le \Lambda_0 
\le \Lambda.
&
\label{bound_by_K_n}
\end{eqnarray}
\eqref{lem:boundedness_of_sequences:IZ-torsion}
We conclude from 
Lemma~\ref{lem:Fuglede-Kadison_and_tors_for_G_is_trivial}
and~\eqref{bound_by_K_n} 
\begin{eqnarray}
\frac{\ln\left(\bigl|\tors(H_n(C[i]_*))\bigr|\right)}{[G:G_i]}
& = &
\frac{\ln\left(\bigl|\tors(\coker(c[i]_{n+1}))\bigr|\right)}{[G:G_i]}
\label{bound_for_torsH_n_by_K_n}
\\
 & \le &
\frac{\left|\ln\bigl({\det}_{\caln(\{1\})}\bigl(c[i]_{n+1}^{(2)}\bigr)\bigr)\right|}{[G:G_i]}
\nonumber
\\
& \le &
\Lambda_0.
\nonumber
\end{eqnarray}
This implies for all $i \in I$ and $n \ge 0$
\[
    0 \le \frac{\ln\left(\bigl|\tors(H_n(C[i]_*))\bigr|\right)}{[G:G_i]} \le \Lambda.
 \]
\eqref{lem:boundedness_of_sequences:alpha} 
 Let $\Delta_n \colon C_n \to C_n$ be
 the combinatorial Laplacian of the finite based free $\IZ G$-chain complex
 $C_*$. Then $\Delta[i]_n \colon C[i]_n \to C[i]_n$ is the combinatorial
 Laplacian of the finite based free $\IZ$-chain complex $C[i]_*$, and
 $\Delta[i]_n^{(2)} \colon C[i]_n^{(2)} \to C[i]_n^{(2)}$ is the same as the
 Laplacian on the $\caln(\{1\})$-Hilbert chain complex $C[i]_*^{(2)}$. Let 
 $j[i] \colon \ker(\Delta[i]_n) \to C[i]_n$ be the inclusion. It induces a map
 $j[i]^{(2)} \colon \ker(\Delta[i]_n)^{(2)} \to C[i]_n^{(2)}$ whose image is
 contained in $\ker\bigl(\Delta[i]_n^{(2)}\bigr)$.

Obviously we have $\ker(c[i]_n) \cap \ker\bigl(c[i]_{n+1}^*\bigr) \subseteq \ker(\Delta[i]_n)$. 
If $k \cdot x \in \ker(c[i]_n) \cap \ker\bigl(c[i]_{n+1}^*\bigr)$ for $k \in \IZ, k \not=  0$ and 
$x \in \ker(\Delta[i]_n)$, then  $x \in \ker(c[i]_n) \cap \ker\bigl(c[i]_{n+1}^*\bigr)$ 
since $C[i]_*$ is free as $\IZ$-module.  Since 
$\ker\bigl(\Delta[i]_n^{(2)}\bigr) = \ker\bigl(c[i]_n^{(2)}\bigr) \cap  \ker\bigl((c[i]_{n+1}^{(2)})^*\bigr)$ 
(see~\cite[Lemma~1.18 on page~24]{Lueck(2002)}), the finitely generated free abelian groups 
$\ker(c[i]_n) \cap \ker\bigl(c[i]_{n+1}^*\bigr)$
and $\ker(\Delta[i]_n)$ have the same rank. This implies
 \begin{eqnarray}
   \ker(\Delta[i]_n)  & = & \ker(c[i]_n) \cap \ker\bigl(c[i]_{n+1}^*\bigr).
   \label{lem:boundedness_of_sequences:alpha:(1)}
 \end{eqnarray}
Let $u[i] \colon \ker(\Delta[i]_n) \to H_n(C[i]_*)_f$
be the composite of the inclusion $ \ker(\Delta[i]_n) \to \ker(c[i]_n)$ 
with the obvious projection $\ker(c[i]_n) \to H_n(C[i]_*)_f$.
Let $v[i] \colon \ker\bigl(\Delta[i]_n^{(2)}\bigr) \to H_n^{(2)}\bigl(C[i]_*^{(2)}\bigr)$
be the composite of the inclusion $ \ker\bigl(\Delta[i]_n^{(2)}\bigr) \to \ker\bigl(c[i]_n^{(2)}\bigr)$
and the projection $\ker\bigl(c[i]_n^{(2)}\bigr)  \to H_n^{(2)}\bigl(C[i]_*^{(2)}\bigr)$.
The following diagram of isomorphisms of $\caln(\{1\})$-Hilbert modules commutes
\[
\xycomsquareminus
{\ker(\Delta[i]_n)^{(2)}}{j[i]^{(2)}}{\ker\bigl(\Delta[i]_n^{(2)}\bigr)}
{u[i]^{(2)}}{v[i]}
{\bigl(H_n(C[i]_*)_f\bigr)^{(2)}}{\alpha[i]_n}{H_n^{(2)}\bigl(C[i]_*^{(2)}\bigr)}
\]
We conclude from~\cite[Theorem~3.14~(1) on page~128]{Lueck(2002)}
\begin{eqnarray}
\quad \quad \quad {\det}_{\caln(\{1\})}(\alpha[i]_n) \cdot {\det}_{\caln(\{1\})}(u[i]^{(2)})
&  = &
{\det}_{\caln(\{1\})}(v[i]) \cdot {\det}_{\caln(\{1\})}(j[i]^{(2)}).
\label{lem:boundedness_of_sequences:alpha:(2)}
\end{eqnarray}
Since $v[i]$ is by definition an isometric isomorphism, we have
\begin{eqnarray}
{\det}_{\caln(\{1\})}(v[i])  & = & 1.
\label{lem:boundedness_of_sequences:alpha:(3)}
\end{eqnarray}
We conclude from Lemma~\ref{lem:Fuglede-Kadison_and_tors_for_G_is_trivial} 
\begin{eqnarray}
& 1 \le {\det}_{\caln(\{1\})}(j[i]^{(2)})  \le  
{\det}_{\caln(\{1\})}\bigl(\Delta[i]_n^{(2)}\bigr);&
\label{lem:boundedness_of_sequences:alpha:(4)}
\\
 & {\det}_{\caln(\{1\})}(u[i]^{(2)})
 =   
|\coker(u[i])|. &
\label{lem:boundedness_of_sequences:alpha:(5)}
\end{eqnarray}
The following diagram commutes
\[
\xymatrix{
\ker(c[i]_n) \cap \ker\bigl(c[i]_{n+1}^*\bigr) \ar[r]^-{k[i]_1} \ar[d]^{l[i]_1}
& 
\ker(c[i]_n)  \ar[r]^-{q[i]_1} \ar[d]^{l[i]_2} 
&
\ker(c[i]_n)/\overline{\im(c[i]_{n+1})} \ar[d]^{l[i]_3}
\\
 \ker\bigl(c[i]_{n+1}^*\bigr) \ar[r]^-{k[i]_2} 
& 
C[i]_n \ar[r]^-{q[i]_2} 
&
C[i]_n/\overline{\im(c[i]_{n+1})}
}
\]
where the maps $k[i]_m$ and $l[i]_m$ are inclusions 
and the maps $q[i]_m$ are projections and $\overline{\im(c[i]_{n+1})}$
has been introduced in Notation~\ref{not:closure_of_submodule}.
An easy diagram chase using $\overline{\im(c[i]_{n+1})} \subseteq \ker(c[i]_n)$ 
shows that $l[i]_3$ induces
a map $\overline{l[i]_3} \colon \coker\bigl(q[i]_1 \circ k[i]_1) 
\to \coker(q[i]_2 \circ k[i]_2\bigr)$ such that
$\overline{l_3}$ is injective. This implies
\begin{eqnarray}
\left|\coker(q[i]_1 \circ k[i]_1))\right| 
& \le & 
\left|\coker(q[i]_2 \circ k[i]_2)\right|.
\label{lem:boundedness_of_sequences:alpha:(6)}
\end{eqnarray}
The map $u[i]$ agrees with $q[i]_1 \circ k[i]_1$. We conclude
from~\eqref{lem:boundedness_of_sequences:alpha:(5)} 
and~\eqref{lem:boundedness_of_sequences:alpha:(6)}
\begin{eqnarray}
 {\det}_{\caln(\{1\})}(u[i]^{(2)})
& = & 
|\coker(u[i])|
\label{lem:boundedness_of_sequences:alpha:(7)}
\\
& = & 
\left|\coker(q[i]_1 \circ k[i]_1)\right| 
\nonumber
\\
& \le & 
\left|\coker(q[i]_2 \circ k[i]_2\bigr)\right|.
\nonumber
\end{eqnarray}
Let $\iota[i] \colon \ker(c[i]_{n+1}^*) \oplus \overline{\im(c[i]_{n+1})} \to C[i]_n$ 
be the inclusion, $k[i]_3 \colon \overline{\im(c[i]_{n+1})} \to C[i]_n$ 
be the inclusion, and $q[i]_3 \colon C[i]_n \to C[i]_n/\ker(c[i]_{n+1}^*)$ 
be the projection. The maps $\iota[i]$ and $q[i]_3 \circ k[i]_3$
are injective and have finite cokernels.
We conclude from Lemma~\ref{lem:Fuglede-Kadison_and_tors_for_G_is_trivial} 
\begin{eqnarray}
\bigl|\coker(q[i]_2 \circ k[i]_2\bigr|
& = & 
\big|\coker(\iota[i])\bigr|
\label{lem:boundedness_of_sequences:alpha:(8)}
\\
& = & 
\bigl|\coker(q[i]_3 \circ k[i]_3)\bigr|
\nonumber
\\
& = & 
{\det}_{\caln(\{1\})}\left(\bigl(q[i]_3 \circ k[i]_3\bigr)^{(2)}\right).
\nonumber
\end{eqnarray}
The map $c[i]_{n+1}^* \colon C[i]_n \to C[i]_{n+1}$ induces an injection
$\overline{c[i]_{n+1}^*} \colon C[i]_n/\ker(c[i]_{n+1}^*) \to C[i]_{n+1}$.
Let $\overline{c[i]_{n+1}} \colon C[i]_{n+1} \to \overline{\im(c[i]_{n+1})}$ 
be the map induced by $c[i]_{n+1}$.
Then the composite
\[
C[i]_{n+1} \xrightarrow{\overline{c[i]_{n+1}}} \overline{\im(c[i]_{n+1})} 
\xrightarrow{k[i]_3} C[i]_n \xrightarrow{q[i]_3} 
C[i]_n/\ker(c[i]_{n+1}^*) \xrightarrow{\overline{c[i]_{n+1}^*}} C[i]_{n+1}
\]
agrees with the composite $c[i]_{n+1}^* \circ c[i]_{n+1}$. 
We conclude from~\cite[Theorem~3.14~(1)  on page~128 and  Lemma~3.15~(4)  on page~129]{Lueck(2002)}
\begin{eqnarray*}
\lefteqn{\left({\det}_{\caln(\{1\})}\bigl(c[i]_{n+1}^{(2)}\bigr)\right)^2}
& & 
\\
& = & 
{\det}_{\caln(\{1\})}\bigl((c[i]_{n+1}^{(2)})^* \circ c[i]_{n+1}^{(2)}\bigr)
\\
& = & 
{\det}_{\caln(\{1\})}\left(\bigl(c[i]_{n+1}^*\circ c[i]_{n+1}\bigr)^{(2)}\right)
\\
& = & 
{\det}_{\caln(\{1\})}\left(\bigl(\overline{c[i]_{n+1}^*} \circ (q[i]_3 
\circ k[i]_3) \circ \overline{c[i]_{n+1}}\bigr)^{(2)}\right)
\\
& = & 
{\det}_{\caln(\{1\})}\left(\bigl(\overline{c[i]_{n+1}^*} \bigr)^{(2)}
\circ \bigl(q[i]_3 \circ k[i]_3\bigr)^{(2)} 
\circ \bigl(\overline{c[i]_{n+1}}\bigr)^{(2)}\right)
\\
& = & 
{\det}_{\caln(\{1\})}\left(\bigl(\overline{c[i]_{n+1}^*} \bigr)^{(2)}\right)
\cdot
{\det}_{\caln(\{1\})}\left(\bigl(q[i]_3 \circ k[i]_3\bigr)^{(2)}\right)
\cdot
{\det}_{\caln(\{1\})}\left(\bigl(\overline{c[i]_{n+1}}\bigr)^{(2)}\right).
\end{eqnarray*}
Since $1 \le {\det}_{\caln(\{1\})}\left(\bigl(\overline{c[i]_{n+1}^*} \bigr)^{(2)}\right)$ 
and 
$1 \le {\det}_{\caln(\{1\})}\left(\bigl(\overline{c[i]_{n+1}}\bigr)^{(2)}\right)$,
this implies
\begin{eqnarray}
 {\det}_{\caln(\{1\})}\left(\bigl(q[i]_3 \circ k[i]_3\bigr)^{(2)}\right)
& \le & 
\left({\det}_{\caln(\{1\})}\bigl(c[i]_{n+1}^{(2)}\bigr)\right)^2.
\label{lem:boundedness_of_sequences:alpha:(9)}
\end{eqnarray}
We derive  from~\eqref{lem:boundedness_of_sequences:alpha:(7)},~%
\eqref{lem:boundedness_of_sequences:alpha:(8)} 
and~\eqref{lem:boundedness_of_sequences:alpha:(9)}
\begin{eqnarray}
& 
1 \le  {\det}_{\caln(\{1\})}\bigl(u[i]^{(2)}\bigr) \le  
\left({\det}_{\caln(\{1\})}\bigl(c[i]_{n+1}^{(2)}\bigr)\right)^2.
&
\label{lem:boundedness_of_sequences:alpha:(10)}
\end{eqnarray}
We conclude from~\eqref{lem:boundedness_of_sequences:alpha:(2)},~%
\eqref{lem:boundedness_of_sequences:alpha:(3)},%
~\eqref{lem:boundedness_of_sequences:alpha:(4)},
and~\eqref{lem:boundedness_of_sequences:alpha:(10)} 
\begin{eqnarray}
& \left({\det}_{\caln(\{1\})}\bigl(c[i]_{n+1}^{(2)}\bigr)\right)^{-2}
\le {\det}_{\caln(\{1\})}(\alpha[i]_n) \le
{\det}_{\caln(\{1\})}\bigl(\Delta[i]_n^{(2)}\bigr).
\label{lem:boundedness_of_sequences:alpha:(11)} 
\end{eqnarray}
We conclude 
from~\cite[Lemma~1.18 on page~24 and Lemma~3.15~(4) and (7) on page~129 and~130]{Lueck(2002)}
(see also the proof of~\cite[Lemma~3.30 on page~140]{Lueck(2002)})
\begin{multline}
\ln\bigl({\det}_{\caln(\{1\}}(\Delta[i]_n^{(2)})\bigr)
\\ = \;
2 \cdot \left(\ln\bigl({\det}_{\caln(\{1\}}(c[i]_n^{(2)})\bigr) 
+ \ln\bigl({\det}_{\caln(\{1\}}(c[i]_{n+1}^{(2)})\bigr)\right).
\label{det(Delta)_and_det(c)}
\end{multline}
Now~\eqref{bound_by_K_n} and~\eqref{det(Delta)_and_det(c)} imply for all $n \ge 0$
\begin{eqnarray*}
\frac{\ln\bigl({\det}_{\caln(\{1\}}(\Delta[i]_n)\bigr)}{[G:G_i]}
& \le &
\Lambda;
\\
\frac{\ln\bigl({\det}_{\caln(\{1\}}(c[i]_n)^{-2}\bigr)}{[G:G_i]}
& \ge &
- \Lambda.
\end{eqnarray*}
This implies together with~\eqref{lem:boundedness_of_sequences:alpha:(11)} 
for all $i \in I$ and $n \ge 0$
\begin{eqnarray*}
&
-\Lambda \le \ln\bigl({\det}_{\caln(\{1\})}(\alpha[i]_n)\bigr)
\le
\Lambda.
&
\end{eqnarray*}
This finishes the proof of Lemma~\ref{lem:boundedness_of_sequences}.
\end{proof}

\subsection{The maps $\alpha[i]_n$}
\label{subsec:The_maps_alpha[i]_n}

In this subsection we deal in a slightly more general context with the maps $\alpha[i]_n$.
Notice that in the lemma below the group $N$ is not required to be abelian, the 
condition about the $N$-action is on homology with $\IQ$-coefficients
and the number $n$ is fixed.

\begin{remark}[Strategy]
  \label{rem:strategy}
  The proof of the next result will reveal our main strategy and idea.  We will
  consider the quotient $Q = G/N$. We will use the a priori bounds above
  applied to everything on the $Q$-level for the system $\{Q_i \mid i \in I\}$,
  where $Q_i$ is the image of $G_i$ under the projection $G \to Q$.  We will
  show that the values on the $G$-level and $Q$ level differ in a controlled
  manner. Finally we use the facts that $[G:G_i] = [N:(N \cap G_i)] \cdot
  [Q:Q_i]$ and $\lim_{i \to \infty} [N:(N \cap G_i)] = \infty$.
\end{remark}

\begin{lemma} \label{lem:alpha_n_does_not_contribute}
  Let $C_*$  be a based free $\IZ G$-chain complex. Fix $n \ge 0$.  
  Assume that there is an infinite normal subgroup 
  $N \subseteq G$ and an index $i_0$ such that $N$ acts trivially on
  $\IQ \otimes_{\IZ} H_n(C[i]_*)$ for every $i \ge i_0$. Then:
  \[
   \lim_{i \in I} \left| \frac{\ln\bigl({\det}_{\caln(\{1\})}(\alpha[i]_n)\bigr)}
   {[G:G_i]}\right| = 0.
  \]
\end{lemma}
\begin{proof} 
  We can assume without loss of generality $G = G_{i_0}$ and that $N$ operates trivially
 on $H_n( \IQ\otimes_{\IZ} C[i]_*)$ for all $i \in I$, since $C[i]_*$ agrees
  as based free $\IZ$-chain complex with $ \IZ[G_{i_0}/G_i] \otimes_{\IZ[G_{i_0}]} C_* 
 $ and $[G:G_i] = [G_{i_0} : G_i] \cdot [G:G_{i_0}]$ holds for
$i \ge i_0$ and $N \cap G_{i_0}$ is an infinite normal subgroup of $G_{i_o}$.

Put $Q = G/N$. Let $\pr \colon G \to Q$ be the projection. Put $Q_i = \pr(G_i)$.
Then we obtain an exact sequence of finite groups $1 \to N/(N \cap G_i) \to G/G_i \to Q/Q_i \to 1$.
This implies for every $i \in I$
\begin{eqnarray}
[G:G_i] & = & [N :( N \cap G_i)] \cdot [Q:Q_i].
\label{estimate_for_[G:G_i|_byN_and_Q_here}
\end{eqnarray}
Define the finite based free $\IZ Q$-chain complex $D_*$ 
by $D_* = \IZ Q\otimes_{\IZ G} C_* $.
Define the based free finite $\IZ[Q/Q_i]$-chain complex $D[i]_*$ by
$D[i]_* = \IZ[Q/Q_i] \otimes_{\IZ Q} D_*  = \IZ[Q/Q_i] \otimes_{\IZ[G/G_i]} C[i]_* $.
We get $D[i]_* = \IZ\otimes_{\IZ [N/(N \cap G_i)]}  C[i]_* $ 
as based free $\IZ$-chain complexes. 

Let $\pr[i]_* \colon C[i]_* \to D[i]_*$ be the canonical projection.  Let
$j[i]_* \colon \bigl(C[i]_*\bigr)^{N/(N \cap G_i)} 
\to D[i]_* = \IZ \otimes_{\IZ [N/(N \cap G_i)]} C[i]_*$ be the obvious chain map.  Let $E[i]_*$ be
its cokernel.  Each $E[i]_n$ is annihilated by multiplication with 
$|N/(N \cap G_i)|$. Since we obtain a short exact sequence of $\IZ$-chain complexes 
$0 \to \bigl(C[i]_*\bigr)^{N/(N \cap G_i)} \xrightarrow{j[i]_*} D[i]_* \to E[i]_* \to 0$, we
conclude by considering the associated long homology sequence, that the kernel
and the cokernel of the map $H_n(j[i]_*) \colon H_n\bigl((C[i]_*)^{N/(N \cap G_i)}\bigr)
\to H_n\bigl(D[i]_*\bigr)$ is annihilated by multiplication with $|N/(N \cap
G_i)|$. Since $N/(N \cap G_i)$ acts trivially on
$\IQ\otimes_{\IZ} H_n\bigl(C[i]_*\bigr)$, the canonical map
$\IQ\otimes_{\IZ}  H_n\bigl((C[i]_*)^{N/(N \cap G_i)}\bigr) \to \IQ\otimes_{\IZ}  H_n(C[i]_*)$ is bijective. Hence the map 
$H_n(\pr[i]_*)_f \colon H_n(C[i]_*)_f \to H_n(D[i]_*)_f$ 
is injective and its cokernel is annihilated by
multiplication with $|N/(N \cap G_i)|$. This implies together with
Lemma~\ref{lem:rho(2)-rhoZ} applied to $H_n(\pr[i]_*)_f$ viewed as
$1$-dimensional finite free $\IZ$-chain complex
\begin{multline*}
1 \le {\det}_{\caln(\{1\}}\bigl(\bigl(H_n(\pr[i]_*)_f\bigr)^{(2)}\bigr) =
\bigl|\tors\bigl(\coker(H_n(\pr[i]_*)_f)\bigl)\big| 
\\
\le
|N/(N \cap G_i)|^{\rk_{\IZ}(H_n(D[i]_*)_f)}.
\end{multline*}
This together with~\eqref{estimate_for_[G:G_i|_byN_and_Q_here} shows
\begin{eqnarray*}
0 & \le &
\frac{\ln\left({\det}_{\caln(\{1\}}\bigl(\bigl(H_n(\pr[i]_*)_f\bigr)^{(2)}\bigr)\right)}
{[G:G_i]} 
\\
& \le &
\frac{\ln\bigl([N:(N \cap G_i)]\bigr)}{[N:(N \cap G_i)]} 
\cdot \frac{\rk_{\IZ}(H_n(D[i]_*)_f)}{[Q:Q_i]}
\\
& \le &
\frac{\ln\bigl([N:(N \cap G_i)]\bigr)}{[N:(N \cap G_i)]} 
\cdot \frac{\rk_{\IZ}(D[i]_n)}{[Q:Q_i]}
\\
& \le &
\frac{\ln\bigl([N:(N \cap G_i)]\bigr)}{[N:(N \cap G_i)]} 
\cdot \rk_{\IZ Q}(D_n).
\end{eqnarray*}
Since $\lim_{i \in I} [N:(N \cap G_i)] = \infty$ and 
hence $\lim_{i \in I} \frac{\ln\bigl([N:(N \cap G_i)]\bigr)}{[N:(N \cap G_i)]} = 0$, 
we conclude
\begin{eqnarray}
\lim_{i \in I} \; 
\frac{\ln\left({\det}_{\caln(\{1\}}\bigl(\bigl(H_n(\pr[i]_*)_f\bigr)^{(2)}\bigr)\right)}
{[G:G_i]} 
& = & 0.
\label{lim_n_to_infty_H_n(pr[i])}
\end{eqnarray}
The following diagram commutes
\[
\xymatrix@!C=15em{
\bigl(C[i]_n^{(2)}\bigr)^{N/(N \cap G_i)}  
\ar[r]^{\bigl(\Delta[i]_n^{(2)}\bigr)^{N/(N \cap G_i)}} \ar[d]^{k[i]_n} 
&\bigl(C[i]_n^{(2)}\bigr)^{N/(N \cap G_i)}  \ar[d]^{k[i]_n} 
\\
C[i]_n^{(2)}  \ar[r]^{\Delta[i]_n^{(2)}} \ar[d]^{\pr[i]_n^{(2)}} 
& C[i]_n^{(2)} \ar[d]^{\pr[i]_n^{(2)}} 
\\
D[i]_n^{(2)}  \ar[r]^{\Delta[i,D]_n^{(2)}}
& D[i]_n^{(2)} 
}
\]
where $k[i]_n$ is the inclusion, $\Delta[i]_n^{(2)}$ is 
the $n$th Laplacian of $C[i]_*^{(2)}$ and
$\Delta[i,D]_n^{(2)}$ is the $n$th Laplacian of $D[i]_*^{(2)}$.

Recall that we equip $\IC[N/(N \cap G_i)] $ and 
$\IC\otimes_{\IC[N/(N \cap G_i)]} \IC[N/(N \cap G_i)] $ with the Hilbert space
structure coming from the obvious $\IC$-basis. Hence $N/(N \cap G_i) $ and
$\{1 \otimes e\}$ are Hilbert basis for $\IC[N/(N \cap G_i)] $ and 
$\IC\otimes_{\IC[N/(N \cap G_i)]}\IC[N/(N \cap G_i)]$. We equip
$\IC[N/(N \cap G_i)]^{N/(N \cap G_i)} \subseteq \IC[N/(N \cap G_i)]$ 
with the sub Hilbert space structure.
Let $N = \sum_{\overline{k} \in N/(N \cap G_i)} \overline{k} \in \IC[N/(N \cap G_i)]$ 
be the norm element. Then $\{N/\sqrt{[N:(N \cap G_i)]}\}$ is a Hilbert
basis of the Hilbert space $\IC[N/(N \cap G_i)]^{N/(N \cap G_i)}$.  The obvious composite
\[
\IC[N/(N \cap G_i)]^{N/(N \cap G_i)} \to \IC[N/(N \cap G_i)] 
\to \IC\otimes_{\IC[N/(N \cap G_i)]} \IC[N/(N \cap G_i)]
\]
sends $N/\sqrt{[N:(N \cap G_i)]}$ to an element of norm $\sqrt{[N:(N \cap G_i)]}$.
This implies that the composite 
$\pr[i]_n^{(2)} \circ k[i]_n \colon  \bigl(C[i]_n^{(2)}\bigr)^{N/(N \cap G_i)} 
\to D[i]_n^{(2)}$ 
satisfies
\[
||x||_{L^2} \le \bigl|\bigl|\pr[i]_n^{(2)} \circ j[i]_n(x)\bigr|\bigr|_{L^2} 
\quad \text{for}\; x \in \bigl(C[i]_n^{(2)}\bigr)^{N/(N \cap G_i)}.
\] 
We conclude that the induced map
\[\pr[i]_n^{(2)} \circ k[i]_n|_{\ker((\Delta[i]_n^{(2)})^{N/(N \cap G_i)})} \colon
\ker\bigl(\bigl(\Delta[i]_n^{(2)}\bigr)^{N/(N \cap G_i)}\bigr) 
\to \ker\bigl(\Delta[i,D]_n^{(2)}\bigr)
\]
satisfy the corresponding inequality. Notice
that $\bigl(\Delta[i]_n^{(2)}\bigr)^{N/(N \cap G_i)}$ can be 
viewed as the $n$th Laplacian of the Hilbert
chain complex $\big(C[i]_n^{(2)}\bigr)^{N/(N \cap G_i)}$. Hence the map
\[
H_n^{(2)}\bigl(\pr[i]_*^{(2)} \circ k[i]_*\bigr) \colon 
H_n^{(2)}\left(\bigl(C[i]_*^{(2)}\bigr)^{N/(N \cap G_i)}\right)
\to H_n^{(2)}\bigl(D[i]_*^{(2)}\bigr)
\]
satisfies
\[
||y||_{L^2} \le 
\bigl|\bigl|H_n^{(2)}\bigl(\pr[i]_*^{(2)} \circ k[i]_*\bigr)(y)\bigr|\bigr|_{L^2} 
\quad \text{for}\; y \in H_n^{(2)}\left(\bigl(C[i]_*^{(2)}\bigr)^{N/(N \cap G_i)}\right).
\] 
since the obvious maps $\ker\bigl(\bigl(\Delta[i]_n^{(2)}\bigr)^{N/(N \cap G_i)}\bigr)  
\to H_n^{(2)}\left(\bigl(C[i]_*^{(2)}\bigr)^{N/(N \cap G_i)}\right)$ and
$\ker\bigl(\Delta[i,D]_n^{(2)}\bigr) \to H_n^{(2)}\bigl(D[i]_*^{(2)}\bigr)$ are isometries.
This implies that the norm and hence also the Fuglede-Kadison determinant of the inverse
of $H_n^{(2)}\bigl(\pr[i]_*^{(2)} \circ k[i]_*\bigr)$ is bounded by $1$.
Since both $H_n^{(2)}(\pr[i]_*^{(2)})$ and $H_n^{(2)}(k[i]_*)$ are invertible, we conclude 
from~\cite[Theorem~3.14~(1) on page~128]{Lueck(2002)} 
\begin{eqnarray}
1 & \le & {\det}_{\caln(\{1\})}\bigl(H_n^{(2)}(\pr[i]_*^{(2)})\bigr) 
\cdot {\det}_{\caln(\{1\})}\bigl(H_n^{(2)}(k[i]_*)\bigr).
\label{lem:alpha_n_does_not_contribute:(1)}
\end{eqnarray}

The projection $\IC[N/(N \cap G_i)] \to \IC \otimes_{\IC[N/(N \cap G_i)} \IC[N/(N \cap G_i)]$ 
sends each element in $N/(N \cap G_i)$ to an
element of norm $1$. Hence the operator norm of 
$\pr[i]_n^{(2)} \colon  C[i]_n^{(2)}\to D[i]_n^{(2)}$ is bounded by 
$[N:(N\cap G_i]$. This implies that 
the operator norm of $H_n^{(2)}(\pr[i|_*^{(2)}) \colon 
H_n^{(2)}(C[i]_*) \to H_n(D[i]_*)$ is bounded by $[N:(N\cap G_i]$.
We conclude
\begin{eqnarray}
{\det}_{\caln(\{1\})}\bigl(H_n^{(2)}(\pr[i|_*^{(2)})\bigr) 
& \le & 
[N:(N\cap G_i]^{\dim_{\IC}(H_n^{(2)}(D[i]_*^{(2)})}.
\label{lem:alpha_n_does_not_contribute:(2)}
\end{eqnarray}
Since $k[i]_*$ is an isometric embedding, the norm of the induced map
$H_n^{(2)}(k[i]_*)$ is bounded by $1$. Hence
\begin{eqnarray}
{\det}_{\caln(\{1\})}\bigl(H_n^{(2)}(k[i|_*)\bigr) & \le & 1.
\label{lem:alpha_n_does_not_contribute:(3)}
\end{eqnarray}
Putting~\eqref{lem:alpha_n_does_not_contribute:(1)},~%
\eqref{lem:alpha_n_does_not_contribute:(2)}
and~\eqref{lem:alpha_n_does_not_contribute:(3)} together yields
\begin{eqnarray*}
1 & \le & {\det}_{\caln(\{1\})}\bigl(H_n^{(2)}\bigl(\pr[i]_n\bigr) 
\le [N:(N\cap G_i)]^{\dim_{\IC}(H_n^{(2)}(D[i]_*^{(2)})}.
\end{eqnarray*}
The same argument as in the end of the proof of~\eqref{lim_n_to_infty_H_n(pr[i])} shows
\begin{eqnarray}
\lim_{i \in I} \; 
\frac{\ln\left({\det}_{\caln(\{1\}}\bigl(H_n^{(2)}(\pr[i]_*^{(2)})\bigr)\right)}{[G:G_i]} 
& = & 0.
\label{lim_n_to_infty_H_n(2)(pr[i])}
\end{eqnarray}
We obtain from Lemma~\ref{lem:boundedness_of_sequences}~\eqref{lem:boundedness_of_sequences:alpha} applied to
$Q$, $\{Q_i \mid i \in I\}$ and $D_*$ that there exists a constant $\Lambda$
independent of $i \in I$ such that for all $i \in I$
\begin{eqnarray}
\frac{\ln\left({\det}_{\caln(\{1\}}(\alpha[i,D]_n)\right)}{[Q:Q_i]} & \le & \Lambda.
\label{bound_on_alpha[i,D]}
\end{eqnarray}
Since $\lim_{i \in } [N:(N \cap G_i)] = \infty$, 
we derive from~\eqref{estimate_for_[G:G_i|_byN_and_Q_here} and~\eqref{bound_on_alpha[i,D]}
\begin{eqnarray}
\frac{\ln\left({\det}_{\caln(\{1\}}(\alpha[i,D]_n)\right)}{[G:G_i]} & = & 0.
\label{limit_for_alpha[i,D]}
\end{eqnarray}

The following diagram of isomorphisms of $\caln(\{1\})$-Hilbert modules commutes
\[
\xymatrix{\bigl(H_n(C[i]_*)_f\bigr)^{(2)}\ar[r]^{\alpha[i]_n} 
\ar[d]_{\bigl(H_n(\pr[i]_*)_f\bigr)^{(2)}}
&
H_n^{(2)}\bigl(C[i]_*^{(2)}\bigr) \ar[d]^{H_n^{(2)}\bigl(\pr[i]_*^{(2)}\bigr)}
\\
\bigl(H_n(D[i]_*)_f\bigr)^{(2)} \ar[r]_{\alpha[i,D]_n} 
&
H_n^{(2)}\bigl(D[i]_*^{(2)}\bigr) 
}
\]
where $\alpha[i]_n$ and $\alpha[i,D]_n$ respectively is the 
map associated to the finite based free $\IZ$-chain complex
$C[i]_*$ and $D[i]_*$ respectively in~\eqref{iso_alpha}. 
We conclude from~\cite[Theorem~3.14~(1) on page~128]{Lueck(2002)}
\begin{multline*}
\ln\left({\det}_{\caln(\{1\}}(\alpha[i]_n)\right) 
= 
\ln\left({\det}_{\caln(\{1\}}(\alpha[i,D]_n)\right) 
+ \ln\left({\det}_{\caln(\{1\}}\bigl(\bigl(H_n(\pr[i]_*)_f\bigr)^{(2)}\bigr)\right)
\\- \ln\left({\det}_{\caln(\{1\}}\bigl(H_n^{(2)}(\pr[i]_*^{(2)})\bigr)\right).
\end{multline*}
This together with~\eqref{lim_n_to_infty_H_n(pr[i])},~%
\eqref{lim_n_to_infty_H_n(2)(pr[i])}, and~\eqref{limit_for_alpha[i,D]} implies
\begin{eqnarray*}
\lim_{i \in I} \; 
\left|\frac{\ln\bigl({\det}_{\caln(\{1\})}(\alpha[i]_n)\bigr)}{[G:G_i]}\right|
& = & 0.
\end{eqnarray*}
This finishes the proof 
Lemma~\ref{lem:alpha_n_does_not_contribute}.
\end{proof}

\begin{example}
Let $X$ be a connected $CW$-complex with fundamental group $G = \pi_1(X)$.
Since $G$ acts trivially on $H_0(G_i\backslash \widetilde{X})$ for all $i \in I$, 
Lemma~\ref{lem:alpha_n_does_not_contribute} implies
\[
\lim_{i \in I} \left| \frac{\ln\bigl({\det}_{\caln(\{1\})}(\alpha[i]_0)\bigr)}{[G:G_i]}\right| = 0.
\]
Let $M$ be a closed manifold of dimension $d$ with fundamental group $G = \pi_1(X)$.
Let $G_0$ be the subgroup of index $1$ or $2$ of those elements in $G$ which act
orientation preserving on $\widetilde{M}$. Then $G_0$ acts trivially on $H_d(G_i\backslash \widetilde{M})$
for all $i \in I$. We conclude from Lemma~\ref{lem:alpha_n_does_not_contribute}
\[
\lim_{i \in I} \left| \frac{\ln\bigl({\det}_{\caln(\{1\})}(\alpha[i]_d)\bigr)}{[G:G_i]}\right| = 0.
\]
If we additionally assume
that $G_i\backslash \widetilde{M}$ is a rational homology sphere for all $i \in I$, then
Lemma~\ref{lem:rho(2)-rhoZ}  implies
\[
\lim_{i \in I} \; \frac{\rho^{(2)}\bigl(G_i\backslash \widetilde{M})}{[G:G_i]} 
- \frac{\rho^{\IZ}\bigl(G_i\backslash \widetilde{M})}{[G:G_i]} = 0.
\]
\end{example}

\subsection{Proof of a chain complex version}
\label{subsec:Proof_of_a_chain_complex_version}
In this subsection we prove

\begin{proposition} \label{pro:chain_complex_version} Let $C_*$ be a finite
  based free $\IZ G$-chain complex.  Consider a natural number $d$. Assume that
  there is an infinite  abelian normal subgroup $A \subseteq G$, an index $i_0$
  and a natural number $r$ such that $H_n(C[i]_*)$ is a nilpotent $\IZ[A/(A\cap
  G_i)]$-module of filtration length $\le r$ for every $i \ge i_0$ and $n \le
  d$. Then:
  \begin{enumerate}

  \item \label{pro:chain_complex_version:mg}
     We get for all $n \le d$
     \[
     \lim_{i \in I} \frac{\mg(H_n(C_*[i]))}{[G:G_i]} = 0;
     \]
   \item \label{pro:chain_complex_version:b_n}
     We get for all $n \le d$
     \[
     \lim_{i \in I} \frac{\dim_K\bigl(H_n(C_*[i];K)\bigr)}{[G:G_i]} = 0;
     \]   

  \item \label{pro:chain_complex_version:tors_H}
     We get for all $n \le d$
     \[
     \lim_{i \in I} \frac{\ln(|\tors(H_n(C_*[i]))|)}{[G:G_i]} = 0;
     \]
  
  \item \label{pro:chain_complex_version:alpha}
    We get for all $n \le d$ 
    \[
    \lim_{i \in I}  \frac{\left|\ln({\det}_{\caln(\{1\})}(\alpha[i]_n))\right|}{[G:G_i]} = 0;
    \]   
   \item\label{pro:chain_complex_version:rho}
   We get
   \[
   \lim_{i \in I} \;  \rho^{(2)}\bigl(C[i]_*;\caln(G/G_i)\bigr)
   = \lim_{i \in I} \;  \frac{\rho^{\IZ}(C[i]_*)}{[G:G_i]} = 0.
   \]
\end{enumerate}
\end{proposition}
\begin{proof} 
  We can assume without loss of generality $G = G_{i_0}$ and that $H_n(C[i]_*)$
  is a nilpotent $\IZ[A/(A\cap G_i)]$-module of filtration length $\le r$ for
  all $i \in I$ and $p \ge 0$, since $C[i]_*$ agrees as based free $\IZ$-chain
  complex with $\IZ \otimes_{\IZ[G_i]} C_*$ and 
  $[G:G_i] = [G_{i_0} : G_i] \cdot [G:G_{i_0}]$ holds for $i \ge i_0$, 
  and $A \cap G_{i_0}$ is an infinite normal subgroup of $G_{i_0}$.

  Put $Q = G/A$. Let $\pr \colon G \to Q$ be the projection. Put $Q_i =
  \pr(G_i)$.  Then we obtain an exact sequence of finite groups $1 \to A/(A \cap
  G_i) \to G/G_i \to Q/Q_i \to 1$.  This implies for every $i \in I$
  \begin{eqnarray} [G:G_i] & = & [A :( A \cap G_i)] \cdot [Q:Q_i].
    \label{estimate_for_[G:G_i|_byA_and_Q}
  \end{eqnarray}
  Define the finite based free $\IZ Q$-chain complex $D_*$ by 
  $D_* = \IZ Q\otimes_{\IZ G} C_* $.  Define the based free finite $\IZ[Q/Q_i]$-chain
  complex $D[i]_*$ by $D[i]_* = \IZ[Q/Q_i]\otimes_{\IZ Q} D_*  
  = \IZ[Q/Q_i]\otimes_{\IZ[G/G_i]} C[i]_*$.  We get 
  $D[i]_* = \IZ\otimes_{\IZ [A/(A\cap G_i)]} C[i]_* $ 
  as based free $\IZ$-chain complexes.

Let $\pr[i]_* \colon C[i]_* \to D[i]_*$ be the canonical projection. 
We conclude from Proposition~\ref{pro:estimate_for_mg}
(applied to the finite abelian group $A/(A \cap G_i)$)
that there are functions
\begin{eqnarray*}
    C_0,C_1,D_0,D_1 \colon \bigl\{(r,n,p) \in \IN^3\mid p \le n\bigr\} & \to & \IR, \quad
  \end{eqnarray*}
such that the following is true for every $n \ge 0$ and $i \in I$
\begin{eqnarray}
\mg(H_n(C[i]_*)) 
& \le & 
\sum_{p = 0}^n C_0(r,n,p) \cdot \mg\bigl(A/(A\cap G_i)\bigr)^{C_1(r,n,p)} 
\label{mg(H)_estimate} 
\\ 
& & \hspace{3mm} \cdot  \mg\bigl(H_p(D[i]_*\bigr);
\nonumber 
\\
\quad \ln\left(\bigl|\ker(H_n(\pr[i]_*))\bigr|\right)
& \le &
\sum_{p = 0}^n   D_0(r,n,p)  \cdot \ln\bigl(|A/(A\cap G_i)|\bigr)) 
\label{ker(H_n(pr)_estimate}
\\
& & \hspace{3mm} \cdot \mg\bigl(A/(A\cap G_i)\bigr)^{D_1(r,n,p)} 
\cdot \mg\bigl(H_p(D[i]_*)\bigr);
\nonumber
\\
\quad \quad \quad \ln\left(\bigl|\coker(H_n(\pr[i]_*))\bigr|\right)
& \le &
 \sum_{p = 0}^n  D_0(r,n,p) \cdot  \ln\bigl(|A/(A\cap G_i)|\bigr)) 
\label{coker(H_n(pr)_estimate}
\\
& & \hspace{3mm} \cdot \mg\bigl(A/(A\cap G_i)\bigr)^{D_1(r,n,p)} 
\cdot \mg\bigl(H_p(D[i]_*)\bigr).
\nonumber
\end{eqnarray}
We obtain from Lemma~\ref{lem:boundedness_of_sequences} applied to
$Q$, $\{Q_i \mid i \in I\}$ and $D_*$ that there exists a constant $\Lambda$
independent of $i \in I$ such that for all $i \in I$ and $p \ge 0$ we get
\begin{eqnarray}
\frac{\mg\bigl(H_p(D[i]_*)\bigr)}{[Q:Q_i]}
&\le & 
\Lambda;
\label{mg(H_n(D))_bounded_by_Lambda}
\\
\frac{\ln\bigl(\tors(H_p(D[i]_*))\bigr)}{[Q:Q_i]}
&\le & 
\Lambda.
\label{tors(H_n(D)_bounded_by_Lambda}
\end{eqnarray}
We conclude from~\eqref{estimate_for_[G:G_i|_byA_and_Q},~\eqref{mg(H)_estimate},~%
\eqref{ker(H_n(pr)_estimate},~\eqref{coker(H_n(pr)_estimate},~\eqref{mg(H_n(D))_bounded_by_Lambda},
and~\eqref{tors(H_n(D)_bounded_by_Lambda}
\begin{eqnarray}
\frac{\mg(H_n(C[i]_*))}{[G:G_i]} 
& \le & 
\sum_{p = 0}^n \Lambda \cdot C_0(r,n,p) \cdot \frac{\mg\bigl(A/(A\cap G_i)\bigr)^{C_1(r,n,p)}}{[A:(A\cap G_i)]};
\label{Tokyo_1}
\end{eqnarray}
\begin{multline}
\lefteqn{\frac{\ln\left(\bigl|\ker(H_n(\pr[i]_*))\bigr|\right)}{[G:G_i]} }
\label{Tokyo_2}
\\
 \le \;
\sum_{p = 0}^n \Lambda\cdot  D_0(r,n,p)
\cdot \frac{\ln\bigl(|A/(A\cap G_i)|\bigr) \cdot \mg\bigl(A/(A\cap G_i)\bigr)^{D_1(r,n,p)}}{[A:(A\cap G_i)]};
\end{multline}
\begin{multline}
\frac{\ln\left(\bigl|\coker(H_n(\pr[i]_*))\bigr|\right)}{[G:G_i]} 
\label{Tokyo_3}
\\
\le \; \sum_{p = 0}^n \Lambda \cdot  D_0(r,n,p)
\cdot \frac{\ln\bigl(|A/(A\cap G_i)|\bigr)) \cdot \mg\bigl(A/(A\cap G_i)\bigr)^{D_1(r,n,p)}}{[A:(A\cap G_i)]}.
\end{multline}
One shows for every natural number $m$ by induction over $m$ using
L'Hospital's rule
\begin{eqnarray}
\lim_{x\to \infty} \frac{\ln(x)^m}{x}
& = & 0.
\label{upsala_1}
\end{eqnarray}
We conclude from Lemma~\ref{lem:properties_of_mg}~\eqref{lem:properties_of_mg:estimate}
\begin{eqnarray}
\mg\bigl(A/(A\cap G_i)\bigr) 
& \le &
\frac{\ln\bigl([A : (A \cap G_i)]\bigr)}{\ln(2)}.
\label{upsala_2}
\end{eqnarray}
Since $A$ is infinite and $\bigcap_{i \in I} G_i = \{1\}$, we have
\begin{eqnarray}
\lim_{i \to \infty} [A : (A \cap G_i)] 
& = & \infty.
\label{upsala_3}
\end{eqnarray}
We conclude from~\eqref{upsala_1},~\eqref{upsala_2} and~\eqref{upsala_3}
\begin{eqnarray}
\lim_{i \to \infty}  \frac{\mg\bigl(A/(A\cap G_i)\bigr)^{C_1(r,n,p)}}{[A:(A\cap G_i)]}
& = & 
0;
\label{Bonn_1}
\\
\lim_{i \to \infty}  \frac{\ln\bigl(|A/(A\cap G_i)|\bigr)) \cdot \mg\bigl(A/(A\cap G_i)\bigr)^{C_1(r,n,p)}}{[A:(A\cap G_i)]}
& = &
0.
\label{Bonn_2}
\end{eqnarray}
Now~\eqref{Tokyo_1},~\eqref{Tokyo_2},\eqref{Tokyo_3},~\eqref{Bonn_1} and~\eqref{Bonn_2} imply
\begin{eqnarray}
\lim_{i \in I} \;\frac{\mg(H_n(C[i]_*))}{[G:G_i]} 
& = & 
0;
\label{lim_mg}
\\
\lim_{i \in I}\; \frac{\ln\left(\bigl|\ker(H_n(\pr[i]_*))\bigr|\right)}{[G:G_i]} 
& = & 
0;
\label{lim_ker(pr)}
\\
\lim_{i \in I} \; \frac{\ln\left(\bigl|\coker(H_n(\pr[i]_*))\bigr|\right)}{[G:G_i]} 
& = & 
0.
\label{lim_coker(pr)}
\end{eqnarray}
Now assertion~\eqref{pro:chain_complex_version:mg} follows from~\eqref{lim_mg}.
Since by the universal coefficient theorem we have
\[
\dim_K\bigl(H_n(C[i]_*;K)\bigr) \le \mg(H_n(C[i]_*) + \mg(H_{n-1}(C[i]_*)
\]
assertion~\eqref{pro:chain_complex_version:b_n}  follows
from assertion~\eqref{pro:chain_complex_version:mg}.

We conclude from~\eqref{lim_ker(pr)} and~\eqref{lim_coker(pr)}
\begin{eqnarray}
\lim_{i \in I}\; \left|\frac{\ln\left(\bigl|\tors(H_n(C[i]_*))\bigr|\right)}{[G:G_i]} 
- \frac{\ln\left(\bigl|\tors(H_n(D[i]_*))\bigr|\right)}{[G:G_i]}\right|
& = & 
0.
\label{lim_difference_H_n(tors)}
\end{eqnarray}
We conclude from~\eqref{estimate_for_[G:G_i|_byA_and_Q} 
and~\eqref{tors(H_n(D)_bounded_by_Lambda}
\begin{eqnarray}
\lim_{i \in I}\; \frac{\ln\left(\bigl|\tors(H_n(D[i]_*))\bigr|\right)}{[G:G_i]}
& = & 
0.
\label{lim_H_n(tors(D))}
\end{eqnarray}
We derive from~\eqref{lim_difference_H_n(tors)} and~\eqref{lim_H_n(tors(D))}
\begin{eqnarray}
\lim_{i \in I}\; \frac{\ln\left(\bigl|\tors(H_n(C[i]_*))\bigr|\right)}{[G:G_i]}
& = & 
0.
\label{lim_H_n(tors(C))}
\end{eqnarray}
Now assertion~\eqref{pro:chain_complex_version:tors_H} 
follows from~\eqref{lim_H_n(tors(C))}.

Assertion~\eqref{pro:chain_complex_version:alpha} is a special case of
Lemma~\ref{lem:alpha_n_does_not_contribute} since $A/(A\cap G_i)$ is finite and hence for a nilpotent $\IZ[A/(A\cap G_i)]$-module
$M$ the $A/(A \cap G_i)$-action on $\IQ \otimes_{\IZ} M$ is trivial.

Assertion~\eqref{pro:chain_complex_version:rho} follows 
from assertions~\eqref{pro:chain_complex_version:tors_H}
and~\eqref{pro:chain_complex_version:alpha}
together with Lemma~\ref{lem:rho(2)-rhoZ}.

This finishes the proof of Proposition~\ref{pro:chain_complex_version}.
\end{proof}

\subsection{Proof of a $CW$-complex version}
\label{subsec:Proof_of_a_CW-complex_version}
In this subsection we prove

\begin{theorem}[CW-version]
  \label{the:CW-version}
  Let $X$ be a connected $CW$-complex.  Consider an epimorphism 
  $\phi \colon \pi_1(X) \to G$. Let $p \colon \overline{X} \to X$ be
  the associated $G$-covering.  Consider a natural number $d$. Assume that
  there is an infinite  abelian normal subgroup $A \subseteq G$, an index $i_0$
  and a natural number $r$ such that $H_n(G_i\backslash \overline{X})$ is a nilpotent $\IZ[A/(A\cap  G_i)]$-module 
  of filtration length $\le r$ for every $i \ge i_0$ and $n \le d$ and that the $(d+1)$-skeleton of $X$ is finite. 
  Then:

\begin{enumerate}

\item \label{the:CW-version:limit_of_mg}
We get for all $n \le d$
\[
\lim_{i\in I} \frac{\mg\bigl(H_n(G_i\backslash \overline{X})\bigr)}{[G:G_i]} 
= 0;
\]

\item \label{the:CW-version:limit_of_torsion_in_homology}
We get for all $n \le d$
\[
\lim_{i\in I} \frac{\ln\left(\left|\tors\bigl(H_n(G_i\backslash \overline{X})\bigr)\right|\right)}{[G:G_i]} 
= 0;
\]

\item \label{the:CW-version:L2-Betti}
We get for all $n \le d$
\[
0 = b_n^{(2)}\bigl(\overline{X};\caln(G)\bigr) = \lim_{i \to \infty} \frac{b_n(G_i\backslash X;K)}{[G:G_i]} = 0;
\]

\item \label{the:CW-version:limit_of_torsion}
Suppose that $X$ is a  finite connected $CW$-complexes. Then
\[
0 = \lim_{i\in I} \rho^{(2)}\bigl(G_i\backslash \overline{X};\caln(G/G_i)\bigr)
= \lim_{i\in I} \frac{\rho^{\IZ}\bigl(G_i\backslash \overline{X}\bigr)}{[G:G_i]}.
\]

\end{enumerate}
\end{theorem}
\begin{proof} 
Assertions~\eqref{the:CW-version:limit_of_mg},~\eqref{the:CW-version:limit_of_torsion_in_homology}
and~\eqref{the:CW-version:L2-Betti} follows from Proposition~\ref{pro:chain_complex_version} applied to the 
cellular $\IZ G$-chain complex of $\overline{X}$
truncated in dimensions greater or equal to $(d+2)$, since this truncation does
does not change the homology in dimension $\le d$, and from~\cite[Theorem~0.1]{Lueck(1994c)}. 
Assertion~\eqref{the:CW-version:limit_of_torsion}
follows from Proposition~\ref{pro:chain_complex_version} applied to the 
cellular $\IZ G$-chain complex of $\overline{X}$ itself.
\end{proof}

\begin{remark}\label{rem_no_claim_for_differentials}
Notice that in the situation of  Theorem~\ref{the:CW-version} we do \emph{not} claim
$\lim_{i \in I} \frac{\ln(\det(c_n[i]))}{[G:G_i]} = 0$. A counterexample comes
from $S^1 \times Y$ for a simply connected $CW$-complex $Y$ whose homology contains no torsion
 but whose differentials do not all have trivial Fuglede-Kadison determinant.
\end{remark}

\subsection{Finishing the proof of Theorem~\ref{the:fibrations}}
\label{subsec:Finishing_the_roof_of_main_Theorem}

In this section we finish the proof of Theorem~\ref{the:fibrations}.
For this purpose we will need

\begin{lemma} \label{lem:reducing_to_phi_epi}
Consider the situation of Theorem~\ref{the:fibrations}. Let $\widehat{G} \subseteq G$ be the image of
$\phi \colon \pi_1(X) \to G$. Let $\widehat{\phi} \colon \pi_1(X) \to \widehat{G}$ be the epimorphism
induced by $\phi$ and let $i \colon \widehat{G} \to G$ be the inclusion. Put $\widehat{G}_i := \widehat{G} \cap G_i$. Suppose
that Theorem~\ref{the:fibrations} holds for $\widehat{\phi}$ and the directed system $\{\widehat{G}_i \mid i \in I\}$. 

Then  Theorem~\ref{the:fibrations} holds for also for $\phi$ and the directed system $\{G_i \mid i \in I\}$. 
\end{lemma}
\begin{proof}
Recall that $\overline{X} \to X$ is the $G$-covering associated to $\phi \colon \pi_1(X) \to G$.
Let  $\widehat{X} \to X$ be the $\widehat{G}$-covering associated to $\widehat{\phi} \colon \pi_1(X) \to \widehat{G}$.
Then there is a $G$-homeomorphism $G \times_{\widehat{G}} \widehat{X} \xrightarrow{\cong} \overline{X}$.
There is an obvious $G/G_i$-homeomorphism
\[
G_i\backslash \left(G \times_{\widehat{G}} \widehat{X} \right) 
\xrightarrow{\cong} G/G_i \times_{\widehat{G}/\widehat{G}_i} \widehat{G}_i\backslash\widehat{X}.
\]
Hence we obtain a $G/G_i$-homeomorphism
\begin{eqnarray*}
G_i\backslash \overline{X} & \xrightarrow{\cong} & G/G_i \times_{\widehat{G}/\widehat{G}_i} \widehat{G}_i\backslash\widehat{X}.
\end{eqnarray*}
Since $\widehat{G}/\widehat{G}_i$ is a subgroup of $G/G_i$, this yields a homeomorphism
\begin{eqnarray*}
G_i\backslash \overline{X} & \xrightarrow{\cong} & \coprod_{i = 1}^{[G/G_i:\widehat{G}/\widehat{G}_i]}  \widehat{G}_i\backslash\widehat{X}.
\end{eqnarray*}
This implies
\begin{eqnarray}
\label{tor_ind_H_G}
\frac{\ln\bigl(\big|\tors(H_n(G_i\backslash \overline{X}))\bigr|\bigr)}{[G:G_i]}
& = & 
\frac{[G/G_i:\widehat{G}/\widehat{G}_i] \cdot \ln\bigl(\big|\tors(H_n(\widehat{G}_i\backslash \widehat{X}))\bigr|\bigr)}{[G:G_i]}
\\
& = & 
\frac{\ln\bigl(\big|\tors(H_n(\widehat{G}_i\backslash \widehat{X}))\bigr|\bigr)}{[\widehat{G}:\widehat{G}_i]}.
\nonumber
\end{eqnarray}
and analogously using Lemma~\ref{lem:properties_of_mg}~\eqref{lem:properties_of_mg:subadditivity}
\begin{eqnarray}
\label{b_n_ind_H_G}
\frac{b_n(G_i\backslash \overline{X};K)}{[G:G_i]}
& = & 
\frac{b_n(\widehat{G}_i\backslash \widehat{X};K)}{[\widehat{G}:\widehat{G}_i]};
\\
\label{mg_ind_H_G}
\frac{\mg\bigl(H_n(G_i\backslash \overline{X})\bigr)}{[G:G_i]}
& \le  & 
\frac{\mg\bigr(H_n(\widehat{G}_i\backslash \widehat{X})}{[\widehat{G}:\widehat{G}_i]}.
\end{eqnarray}
We conclude from~\eqref{tor_ind_H_G}
\begin{eqnarray}
\frac{\rho^{\IZ}\bigl(G_i\backslash\overline{X}\bigr)}{[G:G_i]}
& = & 
\frac{\rho^{\IZ}\bigl(\widehat{G}_i\backslash\widehat{X}\bigr)}{[\widehat{G}:\widehat{G}_i]}.
\label{rho(Z)_ind_H_G}
\end{eqnarray}

We conclude from~\cite[Theorem~1.35~(10) on page~38 and Theorem~3.93~(6) on page~162]{Lueck(2002)}
\begin{eqnarray}
b_n^{(2)}\bigl(\overline{X},\caln(G)\bigr) 
& = & 
b_n^{(2)}\bigl(\widehat{X},\caln(\widehat{G})\bigr);
\label{b_n(2)_ind_H_G}
\\
\rho^{(2)}\bigl(\overline{X},\caln(G)\bigr) 
& = & 
\rho_n^{(2)}\bigl(\widehat{X},\caln(\widehat{G})\bigr),
\label{rho(2)_ind_H_G}
\\
\rho^{(2)}\bigl(G_i\backslash \overline{X},\caln(G/G_i)\bigr) 
& = & 
\rho_n^{(2)}\bigl(H_i\backslash \widehat{X},\caln(\widehat{G}/\widehat{G}_i)\bigr).
\label{rho(2)_ind_H/H_i_G/G_i}
\end{eqnarray}
Now Lemma~\ref{lem:reducing_to_phi_epi} follows 
from~\eqref{tor_ind_H_G},~\eqref{b_n_ind_H_G},~\eqref{mg_ind_H_G},~\eqref{rho(Z)_ind_H_G},~\eqref{b_n(2)_ind_H_G},~%
\eqref{rho(2)_ind_H_G}, and~\eqref{rho(2)_ind_H/H_i_G/G_i}.
\end{proof}

\begin{proof}[Proof of Theorem~\ref{the:fibrations}]
Because of Lemma~\ref{lem:reducing_to_phi_epi} we can assume in the sequel that
$\phi \colon \pi_1(X) \to G$ is surjective.
\\[1mm]~%
\eqref{the:fibrations:limit_of_mg},~\eqref{the:fibrations:limit_of_torsion_in_homology},~\eqref{the:fibrations:L2-Betti}, 
and~\eqref{the:fibrations:limit_of_torsion} 
Let $A$ be the image of $G_1(F)$  under the composite
$\phi \circ \pi_1(j) \colon \pi_1(F) \to \pi_1(X) \to G$.
We apply for $i \in I$  Lemma~\ref{lem:fibrations} to the fibration $F \to p^{-1}(B_{d+1}) \to B_{d+1}$
and the map 
\[\phi_i \colon \pi_1(p^{-1}(B_{d+1})) \xrightarrow{\pi_1(i_{d+1})} \pi_1(X) 
\xrightarrow{\phi} G   \xrightarrow{\pr_i} G/G_i
\]
where $i_{d+1} \colon p^{-1}(B_{d+1}) \to X$ is the inclusion
and $\pr_i$ is the projection. Since $i_{d+1}$ is  $(d+1)$-connected, we conclude for $n \le d$
that the $\IZ[A/(A\cap G_i]$-module $H_n(G_i\backslash X)$ is nilpotent of filtration length $(d+2)$ for $n \le d$.
Now we apply Theorem~\ref{the:CW-version} taking $r = d+2$.
\\[1mm]~\eqref{the:fibrations:L2-torsion}
In the case that $\pi_1(j) \colon \pi_1(F) \to \pi_1(X)$ is injective, $G = \pi_1(X)$, 
and $\phi \colon \pi_1(X) \to G$ is the identity, the claim follows
from~\cite[Theorem~3.100 on page~166]{Lueck(2002)} provided that
that $b_n^{(2)}\bigl(\widetilde{F};\caln(\pi_1(F))\bigr) = 0$ holds for $n \ge 0$ 
and that $\rho^{(2)}\bigl(\widetilde{F};\caln(\pi_1(F))\bigr) = 0$ is valid. 
The proof carries directly over to the general case provided that
$b_n^{(2)}\bigl(\overline{F};\caln(H)\bigr) = 0$ for $n \ge 0$ 
and $\rho^{(2)}\bigl(\overline{F};\caln(H)\bigr) = 0$ holds
for the image $H$ of $\phi \circ \pi_1(j) \colon \pi_1(F) \to \pi_1(X) \to G$.
We conclude $b_n^{(2)}\bigl(\overline{F};\caln(H)\bigr) = 0$ for $n \ge 0$ 
from assertion~\eqref{the:fibrations:L2-Betti} applied to the case of the fibration
$F \to F \to \pt$ and the epimorphism $\pi_1(F) \to H$. We have $\rho^{(2)}\bigl(\overline{F};\caln(H)\bigr) = 0$
by assumption. This finishes the proof of Theorem~\ref{the:fibrations}.
\end{proof}


\typeout{--------------------------- References  -----------------------------}



\end{document}